\documentclass[10pt, leqno]{amsart}
\usepackage{amsfonts,amssymb}

\usepackage{amsmath}
\usepackage{amsthm}
\usepackage{amsrefs}
\usepackage{qsymbols}
\usepackage{latexsym}
\usepackage{chngcntr}
\usepackage[noadjust]{cite}
\usepackage{paralist}
\usepackage{esint}

\newtheorem{theorem}{Theorem}[section]

\newtheorem{lemma}[theorem]{Lemma}
\newtheorem{proposition}[theorem]{Proposition}

\theoremstyle{definition}
\newtheorem{definition}[theorem]{Definition}
\newtheorem{remark}[theorem]{Remark}
\newtheorem{observation}[theorem]{Observation}
\newtheorem{assumption}[theorem]{Assumption}

\newcommand{\IR}{\mathbb{R}}
\newcommand{\IC}{\mathbb{C}}
\newcommand{\IN}{\mathbb{N}}

\newcommand{\IP}{\mathbb{P}}

\newcommand{\cH}{\mathcal{H}}

\newcommand{\cM}{\mathcal{M}}
\newcommand{\cN}{\mathcal{N}}

\newcommand{\cF}{\mathcal{F}}

\newcommand{\cR}{\mathcal{R}}
\newcommand{\cA}{\mathcal{A}}
\newcommand{\cB}{\mathcal{B}}
\newcommand{\cC}{\mathcal{C}}
\newcommand{\cL}{\mathcal{L}}

\newcommand{\cT}{\mathcal{T}}

\renewcommand{\L}{\mathrm{L}}
\newcommand{\C}{\mathrm{C}}

\renewcommand{\H}{\mathrm{H}}

\renewcommand{\S}{\mathrm{S}}

\newcommand{\fa}{\mathfrak{a}}
\newcommand{\fb}{\mathfrak{b}}

\newcommand{\e}{\mathrm{e}}
\newcommand{\ii}{\mathrm{i}}
\renewcommand{\d}{\mathrm{d}}
\newcommand{\eps}{\varepsilon}
\newcommand{\loc}{\mathrm{loc}}
\renewcommand\Re{\operatorname{Re}}

\newcommand{\Lop}{\mathcal{L}}

\newcommand{\divergence}{\operatorname{div}}

\DeclareMathOperator{\supp}{supp}
\DeclareMathOperator{\dist}{dist}
\DeclareMathOperator{\tr}{tr}
\DeclareMathOperator{\diam}{diam}

\DeclareMathOperator{\Id}{Id}

\DeclareMathOperator{\dom}{\mathcal{D}}

\numberwithin{equation}{section}

\title[The generalized Stokes operator with bounded measurable coefficients]{A non-local approach to the generalized Stokes operator with bounded measurable coefficients}
\author{Patrick Tolksdorf}
\address{Institut f\"ur Mathematik, Johannes Gutenberg-Universit\"at Mainz, Staudingerweg 9, 55099 Mainz, Germany}
\email{tolksdorf@uni-mainz.de}

\subjclass[2010]{}
\date{\today}
\thanks{}

\begin{document}
\begin{abstract}
We establish functional analytic properties of the Stokes operator with bounded measurable coefficients on $\L^p_{\sigma} (\IR^d)$, $d \geq 2$, for $\lvert 1 / p - 1 / 2 \rvert < 1 / d$. These include optimal resolvent bounds and the property of maximal $\L^q$-regularity. We further give regularity estimates on the gradient of the solution to the Stokes resolvent problem with bounded measurable coefficients. As a key to these results we establish the validity of a non-local Caccioppoli inequality to solutions of the Stokes resolvent problem.
\end{abstract}
\maketitle

\section{Introduction}

\noindent This article is concerned with the investigation of the Stokes resolvent problem
\begin{align}
\label{Eq: Resolvent problem}
 \left\{ \begin{aligned}
  \lambda u - \divergence(\mu \nabla u) + \nabla \phi &= f && \text{in } \IR^d, \\
  \divergence(u) &= 0 && \text{in } \IR^d
 \end{aligned} \right.
\end{align}
for $\lambda$ in some complex sector $\S_{\omega} := \{ z \in \IC : \lvert \arg(z) \rvert < \omega \}$ for some suitable $\omega \in (\pi / 2 , \pi)$ depending on $d$ and the coefficients $\mu$. The coefficients $\mu_{\alpha \beta}^{i j}$ are assumed to be essentially bounded and complex valued; ellipticity is enforced by a G\r{a}rding type inequality. The equation~\eqref{Eq: Resolvent problem} is the resolvent equation for the Stokes operator with bounded measurable coefficients, which is formally given by
\begin{align*}
 A u = - \divergence(\mu \nabla u) + \nabla \phi, \quad \divergence(u) = 0 \quad \text{in} \quad \IR^d.
\end{align*}
On the space $\L^2_{\sigma} (\IR^d)$ - the space of solenoidal $\L^2$-integrable vector fields - the operator $A$ can be realized via a densely defined, closed, and sectorial sesquilinear form and thus one immediately derives by Kato's form method that there exists $\omega \in (\pi / 2 , \pi)$ such that the sector $\S_{\omega}$ is contained in the resolvent set $\rho(- A)$ of $- A$. Moreover, the operator $A$ is sectorial, i.e., there exists $C > 0$ such that for all $\lambda \in \S_{\omega}$ it holds with $p = 2$
\begin{align}
\label{Eq: L2 resolvent bounds}
 \|\lambda u\|_{\L^p_{\sigma}} = \| \lambda (\lambda + A)^{-1} \|_{\L^p_{\sigma}} \leq C \| f \|_{\L^p_{\sigma}} \qquad (f \in \L^p_{\sigma} (\IR^d)).
\end{align}

\indent A natural question is whether the estimate~\eqref{Eq: L2 resolvent bounds} has analogues in $\L^p_{\sigma}$-spaces for numbers $p \neq 2$. In the elliptic situation, i.e., if the operator $L u = - \divergence(\mu \nabla u)$ is considered, this question is well understood. Indeed, it is well-known~\cite{Auscher, Blunck_Kunstmann, Davies, Egert, Tolksdorf} that there exists $\eps > 0$ such that for all
\begin{align}
\label{Eq: Elliptic Lp interval}
 \Big\lvert \frac{1}{p} - \frac{1}{2} \Big\rvert < \frac{1}{d} + \eps
\end{align}
resolvent bounds of the form~\eqref{Eq: L2 resolvent bounds} are valid for the operator $L$. Moreover, it is well-known that the condition~\eqref{Eq: Elliptic Lp interval} is sharp among the class of all elliptic systems with bounded measurable coefficients. Indeed, Davies~\cite{Davies_example} constructed in $d \geq 3$ dimensions and for each $p > 2 d / (d - 2)$ coefficients $\mu$ such that the semigroups operators $\e^{- t L}$, $t > 0$, do not even map $\L^p (\IR^d ; \IC^d)$ into itself. Thus, in particular~\eqref{Eq: L2 resolvent bounds} fails for such $p$. To the best knowledge of the author, if the Stokes operator $A$ is concerned, the validity of the resolvent bounds~\eqref{Eq: L2 resolvent bounds} for some $p \neq 2$ is completely unknown. \par
It is well-known that the resolvent bounds~\eqref{Eq: L2 resolvent bounds} are the starting point for the investigation of further functional analytic properties of $A$. An immediate question is whether $A$ has the property of maximal $\L^q$-regularity and whether its $\H^{\infty}$-calculus is bounded. These properties are of great interest for the study of fractional powers of $A$ and eventually for well-posedness results of nonlinear problems. In particular, rough coefficients are of great interest in the investigation of non-Newtonian fluids in the regime of rough data. \par
If the coefficients $\mu$ are smooth enough, then all these properties mentioned above were established by Pr\"uss and Simonett~\cite{Pruss_Simonett} and Pr\"uss~\cite{Pruss} for all $1 < p < \infty$. Moreover, Solonnikov~\cite{Solonnikov} investigated operators that were not in divergence form. However, there is a big methodological difference between the smooth situation and the situation of mere essentially bounded coefficients as techniques like freezing the coefficients and arguing for variable coefficients via perturbation become unavailable. In the elliptic and rough situation variational techniques replace the method of freezing the coefficients. Indeed, employing Davies' method one establishes so-called off-diagonal estimates for the semigroup $(\e^{- t L})_{t \geq 0}$ which eventually imply the desired resolvent bounds in $\L^p$ for $p$ satisfying~\eqref{Eq: Elliptic Lp interval}, see~\cite{Auscher, Blunck_Kunstmann, Davies}. Another method is by using Caccioppoli's inequality for the resolvent equations $\lambda u + L u = f$ combined with Sobolev's embedding theorem to deduce the validity of certain weak reverse H\"older estimates that - by virtue of an $\L^p$-extrapolation theorem of Shen~\cite{Shen} - also imply the resolvent bounds in $\L^p$ for $p$ satisfying~\eqref{Eq: Elliptic Lp interval}, see~\cite{Tolksdorf}. \par
The use of these variational methods is again problematic if the Stokes operator $A$ is considered. The reason is that the derivation of off-diagonal estimates as well as of the Caccioppoli inequality rely on testing the resolvent equation by some appropriate function multiplied by a cut-off function. Due to the multiplication by the cut-off function, the test function is not divergence-free anymore so that the pressure appears in the inequalities and has to be treated. Recently, Chang and Kang~\cite{Chang_Kang} proved that it is impossible to establish a parabolic Caccioppoli inequality for the Stokes system on the half-space that has the same form as its elliptic counterpart. Up to now, there has not been a satisfactory way of how to handle this pressure term and the purpose of this work is provide an argument for the treatment of the pressure. \par
As the pressure embodies the non-local part in the resolvent problem~\eqref{Eq: Resolvent problem} some non-local terms will enter the inequalities. Kuusi, Mingione, and Sire investigated in~\cite{Kuusi_Mingione_Sire} non-local elliptic integrodifferential operators of fractional type and established a non-local Caccioppoli inequality in this situation. Inspired by their paper, the author extended in~\cite{Tolksdorf_nonlocal} the validity of their non-local Caccioppoli inequality to the resolvent equation of such operators and further extended the proof of Shen's $\L^p$-extrapolation theorem to these non-local estimates. In this paper, we will proceed similarly and establish a non-local Caccioppoli inequality for the Stokes resolvent problem~\eqref{Eq: Resolvent problem}. As this non-local Caccioppoli inequality is only valid for solenoidal right-hand sides $f$, we further need to adapt the extrapolation argument of Shen as this requires general right-hand sides in $\L^2 (\IR^d ; \IC^d)$. \par
Let us introduce some notation to state the main results: 

\begin{assumption}
\label{Ass: Coefficients}
The coefficients $\mu = (\mu_{\alpha \beta}^{i j})_{\alpha , \beta , i , j = 1}^d$ with $\mu_{\alpha \beta}^{i j} \in \L^{\infty} (\IR^d ; \IC)$ for all $1 \leq \alpha , \beta , i , j \leq d$ satisfy for some $\mu_{\bullet} , \mu^{\bullet} > 0$ the inequalities
\begin{align}
\label{Eq: Ellipticity}
 \Re \sum_{\alpha , \beta , i , j = 1}^d \int_{\IR^d} \mu^{i j}_{\alpha \beta} \partial_{\beta} u_j \overline{\partial_{\alpha} u_i} \, \d x \geq \mu_{\bullet} \| \nabla u \|_{\L^2}^2 \qquad (u \in \H^1 (\IR^d ; \IC^d))
\end{align}
and
\begin{align}
\label{Eq: Boundedness}
 \max_{1 \leq i , j , \alpha , \beta \leq d} \| \mu^{i j}_{\alpha \beta} \|_{\L^{\infty}} \leq \mu^{\bullet}.
\end{align}
\end{assumption} 

The operator $A$ is realized on $\L^2_{\sigma} (\IR^d)$ as follows. Let $\H^1_{\sigma} (\IR^d) := \{ f \in \H^1 (\IR^d ; \IC^d) : \divergence(f) = 0 \}$. Define the sesquilinear form
\begin{align*}
 \fa : \H^1_{\sigma} (\IR^d) \times \H^1_{\sigma} (\IR^d) \to \IC, \quad (u , v) \mapsto \sum_{\alpha , \beta , i , j = 1}^d \int_{\IR^d} \mu_{\alpha \beta}^{i j} \partial_{\beta} u_j \overline{\partial_{\alpha} v_i} \, \d x
\end{align*}
and define the domain of $A$ on $\L^2_{\sigma} (\IR^d)$ as
\begin{align*}
 \dom(A) := \bigg\{ u \in \H^1_{\sigma} (\IR^d) : \, \exists f \in \L^2_{\sigma} (\IR^d) \text{ such that } \forall v \in \H^1_{\sigma} (\IR^d) \text{ it holds } \fa (u , v) = \int_{\IR^d} f \cdot \overline{v} \, \d x  \bigg\}.
\end{align*}
Now, for $u \in \dom(A)$ define $A u := f$. The non-local Caccioppoli inequality we prove reads as follows.

\begin{theorem}
\label{Thm: Non-local Caccioppoli}
Let $\mu$ satisfy Assumption~\ref{Ass: Coefficients} for some constants $\mu_{\bullet} , \mu^{\bullet} > 0$. Then there exists $\omega \in (\pi / 2 , \pi)$ such that for all $\theta \in (0 , \omega)$ and all $0 < \nu < d + 2$ there exists $C > 0$ such that for all $\lambda \in \S_{\theta}$, $f \in \L^2_{\sigma} (\IR^d)$, $F \in \L^2 (\IR^d ; \IC^{d \times d})$ the solution $u \in \H^1_{\sigma} (\IR^d)$ to
\begin{align*}
 \lambda \int_{\IR^d} u \cdot \overline{v} \, \d x + \fa (u , v) = \int_{\IR^d} f \cdot \overline{v} \, \d x - \sum_{\alpha , \beta = 1}^d \int_{\IR^d} F_{\alpha \beta} \, \overline{\partial^{\alpha} v_{\beta}} \, \d x \qquad (v \in \H^1_{\sigma} (\IR^d))
\end{align*}
satisfies for all balls $B = B(x_0 , r)$ and all sequences $(c_k)_{k \in \IN_0}$ with $c_k \in \IC^d$
\begin{align*}
 &\lvert \lambda \rvert \sum_{k = 0}^{\infty} 2^{- \nu k} \int_{B(x_0 , 2^k r)} \lvert u \rvert^2 \, \d x + \sum_{k = 0}^{\infty} 2^{- \nu k} \int_{B(x_0 , 2^k r)} \lvert \nabla u \rvert^2 \, \d x \\
 &\qquad \leq \frac{C}{r^2} \sum_{k = 0}^{\infty} 2^{- (\nu + 2) k} \int_{B(x_0 , 2^{k + 1} r)} \lvert u + c_k \rvert^2 \, \d x + \lvert \lambda \rvert \sum_{k = 0}^{\infty} \lvert c_k \rvert 2^{- \nu k} \int_{B(x_0 , 2^{k + 1} r)} \lvert u \rvert \, \d x \\
 &\qquad\qquad + \frac{C}{\lvert \lambda \rvert} \sum_{k = 0}^{\infty} 2^{- \nu k} \int_{B(x_0 , 2^{k + 1} r)} \lvert f \rvert^2 \, \d x + C \sum_{k = 0}^{\infty} 2^{- \nu k} \int_{B(x_0 , 2^{k + 1} r)} \lvert F \rvert^2 \, \d x.
\end{align*}
The constant $\omega$ only depends on $\mu_{\bullet}$, $\mu^{\bullet}$, and $d$ and $C$ depends on $\mu_{\bullet}$, $\mu^{\bullet}$, $d$, $\theta$, and $\nu$.
\end{theorem}

As described above, the non-local Caccioppoli inequality allows to establish resolvent bounds in $\L^p_{\sigma} (\IR^d)$. More precisely, we have the following result.

\begin{theorem}
\label{Thm: Resolvent}
Let $\mu$ satisfy Assumption~\ref{Ass: Coefficients} for some constants $\mu_{\bullet} , \mu^{\bullet} > 0$. There exist $\omega \in (\pi / 2 , \pi)$ such that for all $p$ satisfying
\begin{align}
\label{Eq: Lp interval}
\Big\lvert \frac{1}{p} - \frac{1}{2} \Big\rvert < \frac{1}{d}
\end{align}
and all $\theta \in (0 , \omega)$ there exists $C > 0$ such that for all $\lambda \in \S_{\theta}$ is holds
\begin{align*}
 \| \lambda (\lambda + A)^{-1} f \|_{\L^p_{\sigma}} \leq C \| f \|_{\L^p_{\sigma}} \qquad (f \in \L^2_{\sigma} (\IR^d) \cap \L^p_{\sigma} (\IR^d)).
\end{align*}
The constant $C > 0$ depends only on $d$, $\mu_{\bullet}$, $\mu^{\bullet}$, $p$, and $\theta$. The constant $\omega$ depends only on $d$, $\mu_{\bullet}$, and $\mu^{\bullet}$.
\end{theorem}

Additionally to the $\L^p$-resolvent estimates in Theorem~\ref{Thm: Resolvent} we establish further regularity estimates for solutions to the Stokes resolvent problem.

\begin{theorem}
\label{Thm: Gradient}
Let $\mu$ satisfy Assumption~\ref{Ass: Coefficients} for some constants $\mu_{\bullet} , \mu^{\bullet} > 0$. There exist $\omega \in (\pi / 2 , \pi)$ such that for all $p$ satisfying
\begin{align}
\label{Eq: Interval gradient estimates}
 \frac{2 d}{d + 2} < p \leq 2
\end{align}
and all $\theta \in (0 , \omega)$ there exists $C > 0$ such that for all $\lambda \in \S_{\theta}$ is holds
\begin{align*}
 \lvert \lambda \rvert^{1 / 2} \| \nabla (\lambda + A)^{-1} f \|_{\L^p_{\sigma}} \leq C \| f \|_{\L^p_{\sigma}} \qquad (f \in \L^2_{\sigma} (\IR^d) \cap \L^p_{\sigma} (\IR^d)).
\end{align*}
The constant $C > 0$ depends only on $d$, $\mu_{\bullet}$, $\mu^{\bullet}$, $p$, and $\theta$. The constant $\omega$ depends only on $d$, $\mu_{\bullet}$, and $\mu^{\bullet}$.
\end{theorem}

Theorem~\ref{Thm: Resolvent} allows to realize the operator $A$ as a sectorial operator on the $\L^p_{\sigma}$-spaces for $p$ satisfying~\eqref{Eq: Lp interval}. It is well-known that this is equivalent to the fact that $- A$ generates a bounded analytic semigroup $(\e^{- t A})_{t \geq 0}$ on $\L^p_{\sigma} (\IR^d)$. Additionally, Theorem~\ref{Thm: Gradient} tells us that this semigroup satisfies for $p$ subject to~\eqref{Eq: Interval gradient estimates} gradient estimates of the form
\begin{align*}
 t^{1 / 2} \| \nabla \e^{- t A} f \|_{\L^p} \leq C \| f \|_{\L^p_{\sigma}} \qquad (t > 0 , f \in \L^p_{\sigma} (\IR^d)).
\end{align*}
We mention here the result of Kaplick\'y and Wolf~\cite{Kaplicky_Wolf} who prove a Meyers'-type higher integrability result to obtain even integrability properties for the gradient of the instationary solution for $p$ being slightly larger than $2$. \par
We further prove that the $\L^p_{\sigma}$-realizations of $A$ have the property of maximal $\L^q$-regularity as the following theorem states. 

\begin{theorem}
\label{Thm: Max Reg}
Let $\mu$ satisfy Assumption~\ref{Ass: Coefficients} for some constants $\mu_{\bullet} , \mu^{\bullet} > 0$. Then for all $p$ satisfying~\eqref{Eq: Lp interval} the $\L^p_{\sigma}$-realization of $A$ has maximal $\L^q$-regularity for any $1 < q < \infty$. More precisely, for any $f \in \L^q (0 , \infty ; \L^p_{\sigma} (\IR^d))$ the unique mild solution $u$ to the Cauchy problem
\begin{align*}
\left\{ \begin{aligned}
 u^{\prime} (t) + A u (t) &= f (t), && t > 0, \\
 u (0) &= 0
\end{aligned} \right.
\end{align*}
satisfies $u(t) \in \dom(A)$ for almost every $t > 0$ and $u^{\prime} , A u \in \L^q (0 , \infty ; \L^p_{\sigma} (\IR^d))$ and there exists a constant $C > 0$ depending only on $d$, $\mu_{\bullet}$, $\mu^{\bullet}$, $p$, and $q$ such that
\begin{align*}
 \| u^{\prime} \|_{\L^q (0 , \infty ; \L^p_{\sigma})} + \| A u \|_{\L^q (0 , \infty ; \L^p_{\sigma})} \leq C \| f \|_{\L^q (0 , \infty ; \L^p_{\sigma})}.
\end{align*}
\end{theorem}



We close this introduction by stating some standard notation. Throughout, the space dimension $d$ is assumed to satisfy $d \geq 2$. The natural numbers $\IN$ are given by $\{1 , 2 , \dots \}$ and $\IN_0 := \IN \cup \{ 0 \}$. For a ball $B = B(x_0 , r)$ and some number $\alpha > 0$ we denote by $\alpha B$ the dilated ball $B(x_0 , \alpha r)$. Constants $C > 0$ will be generic and might change its values from line to line. We add subscripts, e.g., $C_d$, $C_{\mu^{\bullet}}$ to indicate the dependence of $C$ of certain quantities. The mean value of a locally integrable function $f$ on a bounded measurable set $\cA$ with $\lvert \cA \rvert > 0$ is denoted by
\begin{align*}
 f_{\cA} := \fint_{\cA} f \, \d x := \frac{1}{\lvert \cA \rvert} \int_{\cA} f \, \d x.
\end{align*}
For the rest of this work, we agree on summing over repeated indices.


\section{Proof of the non-local Caccioppoli inequality}

\noindent Let $B \subset \IR^d$ denote a ball centered in $x_0 \in \IR^d$ with radius $r > 0$. If $u \in \H^1 (2 B)$ is harmonic, then the classical Caccioppoli inequality for $u$ reads as
\begin{align}
\label{Eq: Classical Caccioppoli}
 \int_B \lvert \nabla u \rvert^2 \, \d x \leq \frac{C}{r^2} \int_{2 B} \lvert u \rvert^2 \, \d x,
\end{align}
where $C > 0$ denotes a dimensional constant. Its proof is very simple as it follows after three lines of calculation after testing the equation $- \Delta u = 0$ in $2 B$ by the test function $\eta^2 u$, where $\eta \in \C_c^{\infty} (2 B)$ satisfies $0 \leq \eta \leq 1$, $\eta \equiv 1$ in $B$, and $\| \nabla \eta \|_{\L^{\infty}} \leq 2 / r$. It is well-known that this inequality can be generalized to solutions to elliptic systems in divergence form with bounded measurable coefficients. One can even go further and consider weak solutions $u \in \H^1 (2 B)$ to the equation $\lambda u - \divergence \mu \nabla u = f$ for $\lambda \in \S_{\theta}$. Testing with the same test function as above then delivers the inequality
\begin{align}
\label{Eq: Elliptic Caccioppoli}
 \lvert \lambda \rvert \int_B \lvert u \rvert^2 \, \d x + \int_B \lvert \nabla u \rvert^2 \, \d x \leq \frac{C}{r^2} \int_{2 B} \lvert u \rvert^2 \, \d x + \frac{C}{\lvert \lambda \rvert} \int_{2 B} \lvert f \rvert^2 \, \d x.
\end{align}
Here, the constant $C > 0$ depends on $d$, $\theta$, and the ellipticity and boundedness constants of the coefficients. Clearly, $\theta$ has to be chosen appropriately depending on the ellipticity and boundedness constants of the coefficients. \par
When the Stokes problem with, say, continuous coefficients is concerned, the validity of the classical Caccioppoli inequality~\eqref{Eq: Classical Caccioppoli} was proven by Giaquinta and Modica in~\cite[Thm.~1.1]{Giaquinta_Modica}. The proof bases again on testing the equation by $\eta^2 u$ as above. However, as the test function is not divergence free, the pressure will appear and needs to be handled in the estimates. With a similar argument, Choe and Kozono~\cite{Choe_Kozono} established the Caccioppoli inequality for the Stokes resolvent problem with coefficient matrix $\mu = \Id$ and resolvent parameter $\lambda \in \ii \IR$. In this case, the Caccioppoli inequality reads
\begin{align}
\label{Eq: Odd Stokes Caccioppoli}
 \lvert \lambda \rvert \int_B \lvert u \rvert^2 \, \d x + \int_B \lvert \nabla u \rvert^2 \, \d x \leq \frac{C (1 + \lvert \lambda \rvert r^2)}{r^2} \int_{2 B} \lvert u \rvert^2 \, \d x + \frac{C}{\lvert \lambda \rvert} \int_{2 B} \lvert f \rvert^2 \, \d x.
\end{align}
If one compares the elliptic estimate~\eqref{Eq: Elliptic Caccioppoli} with the estimate for the Stokes resolvent~\eqref{Eq: Odd Stokes Caccioppoli}, one readily sees the difference in the prefactor in front of the $\L^2$-integral of $u$ on the right-hand side. The additional term involving the term $\lvert \lambda \rvert r^2$ results from the treatment of the arising pressure term. Unfortunately, for some purposes it is important that the constant in front of the $\L^2$-integral of $u$ on the right-hand side is \textit{independent} of $\lambda$ and it would be desirable if the constant would simply be $C / r^2$ as in the elliptic case~\eqref{Eq: Elliptic Caccioppoli}. As the pressure reflects to a great extend the \textit{non-local} behavior of the solution, it is, however, not very surprising that something odd happens if the non-local term is ``pressed'' into a local estimate. The goal of this section is take the opposite viewpoint, namely, to prove a non-local counterpart of the Caccioppoli inequality and to recover the prefactor $C / r^2$ in front of the $\L^2$-integrals of $u$ on the right-hand side. The precise result we prove is formulated in Theorem~\ref{Thm: Non-local Caccioppoli}.

To prepare the arguments we first introduce some technical tools. First of all, recall that the Helmholtz projection $\IP$ is given on the whole space by
\begin{align*}
 \IP f = \cF^{-1} \Big[ \Id - \frac{\xi \otimes \xi}{\lvert \xi \rvert^2} \Big] \cF f \qquad \Leftrightarrow \qquad (\Id - \IP) f = \cF^{-1} \frac{\xi \otimes \xi}{\lvert \xi \rvert^2} \cF f.
\end{align*}
Here, $\cF$ denotes the Fourier transform, $f$ denotes an element in $\L^2 (\IR^d ; \IC^d)$, and $\xi \otimes \xi := \xi \xi^{\top}$. Notice that $\IP$ is the orthogonal projection onto $\L^2_{\sigma} (\IR^d)$. In particular, for $f \in \H^1 (\IR^d ; \IC^d)$ one has
\begin{align*}
 \divergence((\Id - \IP) f) = \divergence(f).
\end{align*}
Recall further that $\IP$ and $\Id - \IP$ commute with derivatives whenever the function $f$ is regular enough. Finally, since $\xi \mapsto \frac{\xi \otimes \xi}{\lvert \xi \rvert^2}$ is a Mikhlin symbol, by~\cite[Prop.~VI.4.2]{Stein} there exists a kernel function $k : \IR^d \setminus \{0\} \to \IR^d$ such that
\begin{align}
\label{Eq: Helmholtz kernel 1}
 [(\Id - \IP) f] (x) = \int_{\IR^d} k (x - y) f (y) \, \d y \qquad (f \in \L^2 (\IR^d ; \IC^d) , \, x \notin \supp (f))
\end{align}
and such that there exists $C_d > 0$, depending only on $d$, such that
\begin{align}
\label{Eq: Helmholtz kernel 2}
 \lvert \partial^{\alpha} k (x) \rvert \leq \frac{C_d}{\lvert x \rvert^{d + \lvert \alpha \rvert}} \qquad (x \in \IR^d \setminus \{ 0 \} \text{ and } \alpha \in \IN_0^d \text{ with } \lvert \alpha \rvert \leq 1).
\end{align}
\indent To proceed, let $\L^2_0 (\cA) := \{ f \in \L^2 (\cA) : f_{\cA} = 0 \}$. Let $\cC_1$ denote the annulus $B(0 , 1) \setminus \overline{B(0 , 1 / 2)}$. The Bogovski\u{\i} operator $\cB_1 : \L^2_0 (\cC_1) \to \H^1_0 (\cC_1 ; \IC^d)$ denotes the solution operator to the divergence equation for functions $f \in \L^2_0 (\cC_1)$
\begin{align*}
\left\{ \begin{aligned}
 \divergence(u) &= f && \text{in } \cC_1, \\
 u &= 0 && \text{on } \partial \cC_1.
\end{aligned} \right.
\end{align*}
Thus, we have $\divergence(\cB_1 f) = f$. Furthermore, $\cB_1$ is a bounded operator from $\L^2_0 (\cC_1)$ onto $\H^1_0 (\cC_1 ; \IC^d)$, i.e., there exists a constant $C_{Bog} > 0$ such that
\begin{align*}
 \| \cB_1 f \|_{\H^1 (\cC_1)} \leq C_{Bog} \| f \|_{\L^2 (\cC_1)} \qquad (f \in \L^2_0 (\cC_1)).
\end{align*}
See, e.g., Galdi~\cite[Sect.~III.3]{Galdi} for a construction of this operator. Now, if $\cC_{\alpha}$ denotes the annulus $B(0 , \alpha) \setminus \overline{B(0 , \alpha / 2)}$ for some $\alpha > 0$ and if $f \in \L^2_0 (\cC_{\alpha})$ the rescaled function $f_{\alpha} (x) := \alpha f (\alpha x)$ lies in $\L^2_0 (\cC_1)$. Define the rescaled Bogovski\u{\i} operator on $\cC_{\alpha}$ as
\begin{align*}
 [\cB_{\alpha} f] (x) := [\cB_1 f_{\alpha}] (\alpha^{-1} x) \qquad (f \in \L^2_0 (\cC_{\alpha}) , \, x \in \cC_{\alpha}).
\end{align*}
Clearly, $\cB_{\alpha}$ is bounded from $\L^2_0 (\cC_{\alpha})$ onto $\H^1_0 (\cC_{\alpha} ; \IC^d)$ and satisfies $\divergence \cB_{\alpha} f = f$. Furthermore, by rescaling, the following inequalities holds
\begin{align}
\label{Eq: Homogeneous estimate Bogovskii}
 \| \nabla \cB_{\alpha} f \|_{\L^2 (\cC_{\alpha})} \leq C_{Bog} \| f \|_{\L^2 (\cC_{\alpha})} \qquad (f \in \L^2_0(\cC_{\alpha}))
\end{align}
and
\begin{align}
\label{Eq: Inhomogeneous estimate Bogovskii}
 \| \cB_{\alpha} f \|_{\L^2 (\cC_{\alpha})} \leq \alpha C_{Bog} \| f \|_{\L^2 (\cC_{\alpha})} \qquad (f \in \L^2_0(\cC_{\alpha})).
\end{align}
Finally, let use mention that - with a slight abuse of notation - we will denote the Bogovski\u{\i} operator on annuli not centered in the origin, i.e., on $B (x_0 , \alpha) \setminus \overline{B (x_0 , \alpha / 2)}$ by $\cB_{\alpha}$ as well and notice that~\eqref{Eq: Homogeneous estimate Bogovskii} and~\eqref{Eq: Inhomogeneous estimate Bogovskii} hold with the same constant. \par
We are in the position to prove a lemma on non-local pressure estimates.

\begin{lemma}
\label{Lem: Non-local pressure estimate}
Let the coefficients $\mu$ satisfy~\eqref{Eq: Ellipticity} and~\eqref{Eq: Boundedness} with constants $\mu^{\bullet} , \mu_{\bullet} > 0$. Let $\lambda \in \IC$ and let for $f \in \L^2_{\sigma} (\IR^d)$ and $F \in \L^2 (\IR^d ; \IC^{d \times d})$ the functions $u \in \H^1_{\sigma} (\IR^d)$ and $\phi \in \L^2_{\loc} (\IR^d)$ solve
\begin{align*}
 \left\{ \begin{aligned}
  \lambda u - \divergence \mu \nabla u + \nabla \phi &= f + \divergence(F) && \text{in } \IR^d, \\
  \divergence (u) &= 0 && \text{in } \IR^d
 \end{aligned} \right.
\end{align*}
in the sense of distributions. Let $x_0 \in \IR^d$ and $r > 0$ and define for $k \in \IN$ the annulus $\cC_k := B(x_0 , 2^k r) \setminus \overline{B(x_0 , 2^{k - 1} r)}$. Let $\cC_0$ denote the ball $B(x_0 , r)$. Then there exists a constant $C > 0$ depending only on $\mu^{\bullet}$ and $d$ such that for all $k \in \IN$ we have
\begin{align*}
 &\bigg(\int_{\cC_k} \lvert \phi - \phi_{\cC_k} \rvert^2 \, \d x \bigg)^{\frac{1}{2}} \\
 &\qquad\leq C \bigg( \sum_{\ell = 0}^{k - 2} 2^{\frac{d}{2} (\ell - k)} \big( \| \nabla u \|_{\L^2 (\cC_{\ell})} + \| F \|_{\L^2 (\cC_{\ell})} \big) + \sum_{\substack{\ell \in \IN_0 \\ \lvert \ell - k \rvert \leq 1}} \big( \| \nabla u \|_{\L^2 (\cC_{\ell})} + \| F \|_{\L^2 (\cC_{\ell})} \big) \\
 &\qquad\qquad + \sum_{\ell = k + 2}^{\infty} 2^{(\frac{d}{2} + 1) (k - \ell)} \big( \| \nabla u \|_{\L^2 (\cC_{\ell})} + \| F \|_{\L^2 (\cC_{\ell})} \big) \bigg).
\end{align*}
\end{lemma}

\begin{proof}
Let $k \in \IN$ and let $\cB_{\cC_k} := \cB_{2^k r}$ denote the Bogovski\u{\i} operator on $\cC_k$. By an extension by zero, we view the function $\cB_{\cC_k} ((\phi - \phi_{\cC_k})|_{\cC_k})$ as a function in $\H^1(\IR^d ; \IC^d)$ whose support is contained in $\overline{\cC_k}$. Define the test function
\begin{align*}
 v_k := (\Id - \IP) \cB_{\cC_k} ((\phi - \phi_{\cC_k})|_{\cC_k})
\end{align*}
which is the sum of a function in $\H^1 (\IR^d ; \IC^d)$ with support in $\overline{\cC_k}$ and of a function in $\H^1_{\sigma} (\IR^d)$. Since $u \in \H^1_{\sigma} (\IR^d)$ and $\phi \in \L^2_{\loc} (\IR^d)$ an approximation argument allows to test the resolvent equation by this test function. Because $v_k \in \L^2_{\sigma} (\IR^d)^{\perp}$ and $f \in \L^2_{\sigma} (\IR^d)$ this yields
\begin{align*}
 \int_{\IR^d} \phi \, \divergence ((\Id - \IP) \cB_{\cC_k} ((\overline{\phi - \phi_{\cC_k}})|_{\cC_k})) \, \d x &= \int_{\IR^d} \mu^{i j}_{\alpha \beta} \partial_{\alpha} u_j \partial_{\beta} [(\Id - \IP) \cB_{\cC_k} ((\overline{\phi - \phi_{\cC_k}})|_{\cC_k})]_i \, \d x \\
 &\qquad + \int_{\IR^d} F_{\alpha \beta} \partial_{\alpha} [(\Id - \IP) \cB_{\cC_k} ((\overline{\phi - \phi_{\cC_k}})|_{\cC_k})]_{\beta} \, \d x.
\end{align*}
Notice that we can subtract an arbitrary constant from the left-most pressure since integration is only performed on $\cC_k$ (recall that the divergence applied to $\IP$ is zero and that the support of $\cB_{\cC_k} ((\phi - \phi_{\cC_k})|_{\cC_k})$ is contained in $\overline{\cC_k}$). By our technical preparation in front of this lemma, we derive the validity of the identity
\begin{align*}
 \int_{\cC_k} \lvert \phi - \phi_{\cC_k} \rvert^2 \, \d x &= \int_{\IR^d} \mu^{i j}_{\alpha \beta} \partial_{\alpha} u_j [(\Id - \IP) \partial_{\beta} \cB_{\cC_k} ((\overline{\phi - \phi_{\cC_k}})|_{\cC_k})]_i \, \d x \\
 &\qquad + \int_{\IR^d} F_{\alpha \beta} [(\Id - \IP) \partial_{\alpha} \cB_{\cC_k} ((\overline{\phi - \phi_{\cC_k}})|_{\cC_k})]_{\beta} \, \d x.
\end{align*}
With a constant $C_{d , \mu^{\bullet}} > 0$ depending only on $d$ and $\mu^{\bullet}$, we then find 
\begin{align*}
 \int_{\cC_k} \lvert \phi - \phi_{\cC_k} \rvert^2 \, \d x &\leq C_{d , \mu^{\bullet}} \sum_{\ell = 0}^{\infty} \big( \| \nabla u \|_{\L^2 (\cC_{\ell})} + \| F \|_{\L^2 (\cC_{\ell})} \big) \sum_{\beta = 1}^d \| (\Id - \IP) \partial_{\beta} \cB_{\cC_k} ((\overline{\phi - \phi_{\cC_k}})|_{\cC_k}) \|_{\L^2 (\cC_{\ell})}.
\end{align*}
We proceed as follows: if $\dist(\cC_{\ell} , \cC_k) = 0$ we use the fact that $\Id - \IP$ is an orthogonal projection whose $\L^2$-operator norm is $1$. If $\dist(\cC_{\ell} , \cC_k) > 0$ we use the kernel representation of $\Id - \IP$ stated in~\eqref{Eq: Helmholtz kernel 1} and~\eqref{Eq: Helmholtz kernel 2}. In any case, we will conclude the estimate by employing either~\eqref{Eq: Homogeneous estimate Bogovskii} or~\eqref{Eq: Inhomogeneous estimate Bogovskii} depending on the particular situation. \par
If $\dist(\cC_{\ell} , \cC_k) > 0$, $k \geq 2$, and $\ell \leq k - 2$, then we there is a constant $C_d > 0$ such that
\begin{align*}
 \sum_{\beta = 1}^d \| (\Id - \IP) \partial_{\beta} \cB_{\cC_k} ((\overline{\phi - \phi_{\cC_k}})|_{\cC_k}) \|_{\L^2 (\cC_{\ell})} &\leq C_d \bigg( \int_{\cC_{\ell}} \bigg( \int_{\cC_k} \frac{\lvert \nabla \cB_{\cC_k} ((\overline{\phi - \phi_{\cC_k}})|_{\cC_k}) (y) \rvert}{\lvert x - y \rvert^d} \, \d y \bigg)^2 \, \d x \bigg)^{\frac{1}{2}} \\
 &\leq C_d \frac{\lvert \cC_{\ell} \rvert^{\frac{1}{2}} \lvert \cC_k \rvert^{\frac{1}{2}}}{\dist(\cC_{\ell} , \cC_k)^d} \| \nabla \cB_{\cC_k} ((\overline{\phi - \phi_{\cC_k}})|_{\cC_k}) \|_{\L^2 (\cC_k)} \\
 &\leq C_d C_{Bog} \frac{\lvert \cC_{\ell} \rvert^{\frac{1}{2}} \lvert \cC_k \rvert^{\frac{1}{2}}}{\dist(\cC_{\ell} , \cC_k)^d} \| \phi - \phi_{\cC_k} \|_{\L^2 (\cC_k)}.
\end{align*}
Similarly, if $\dist(\cC_{\ell} , \cC_k) > 0$, $k \geq 1$, and $\ell \geq k + 2$, then it holds that
\begin{align*}
 \sum_{\beta = 1}^d \| (\Id - \IP) \partial_{\beta} \cB_{\cC_k} ((\overline{\phi - \phi_{\cC_k}})|_{\cC_k}) \|_{\L^2 (\cC_{\ell})} &\leq C_d \bigg( \int_{\cC_{\ell}} \bigg( \int_{\cC_k} \frac{\lvert \cB_{\cC_k} ((\overline{\phi - \phi_{\cC_k}})|_{\cC_k}) (y) \rvert}{\lvert x - y \rvert^{d + 1}} \, \d y \bigg)^2 \, \d x \bigg)^{\frac{1}{2}} \\
 &\leq C_d \frac{\lvert \cC_{\ell} \rvert^{\frac{1}{2}} \lvert \cC_k \rvert^{\frac{1}{2}}}{\dist(\cC_{\ell} , \cC_k)^{d + 1}} \| \cB_{\cC_k} ((\overline{\phi - \phi_{\cC_k}})|_{\cC_k}) \|_{\L^2 (\cC_k)} \\
 &\leq C_d C_{Bog} \frac{\lvert \cC_{\ell} \rvert^{\frac{1}{2}} \lvert \cC_k \rvert^{\frac{1}{2}} 2^k r}{\dist(\cC_{\ell} , \cC_k)^{d + 1}} \| \phi - \phi_{\cC_k} \|_{\L^2 (\cC_k)}.
\end{align*}
Moreover, we find
\begin{align*}
 \dist(\cC_{\ell} , \cC_k) \geq 2^{k - 1} r - 2^{\ell} r \geq (2^{k - 1} - 2^{k - 2}) r = 2^{k - 2} r \quad \text{if} \quad k \geq 2 \quad \text{and} \quad \ell \leq k - 2
\end{align*}
and
\begin{align*}
 \dist(\cC_{\ell} , \cC_k) \geq 2^{\ell - 1} r - 2^k r = (2^{\ell - 1} - 2^{\ell - 2}) r = 2^{\ell - 2} r \quad \text{if} \quad  k \geq 1 \quad \text{and} \quad \ell \geq k + 2.
\end{align*}
Combining all the previous estimates delivers in the case $k \geq 2$ and $\ell \leq k - 2$ that
\begin{align*}
 \sum_{\beta = 1}^d \| (\Id - \IP) \partial_{\beta} \cB_{\cC_k} ((\overline{\phi - \phi_{\cC_k}})|_{\cC_k}) \|_{\L^2 (\cC_{\ell})} \leq C_d C_{Bog} 2^{\frac{d}{2} (\ell - k)} \| \phi - \phi_{\cC_k} \|_{\L^2 (\cC_k)}
\end{align*}
and in the case $k \geq 1$ and $\ell \geq k + 2$ that
\begin{align*}
 \sum_{\beta = 1}^d \| (\Id - \IP) \partial_{\beta} \cB_{\cC_k} ((\overline{\phi - \phi_{\cC_k}})|_{\cC_k}) \|_{\L^2 (\cC_{\ell})} \leq C_d C_{Bog} 2^{(\frac{d}{2} + 1)(k - \ell)} \| \phi - \phi_{\cC_k} \|_{\L^2 (\cC_k)}.
\end{align*}
As $C_{Bog}$ also only depends on $d$, we altogether get a constant $C_{d , \mu^{\bullet}} > 0$ that depends only on $d$ and $\mu^{\bullet}$ such that
\begin{align*}
 &\int_{\cC_k} \lvert \phi - \phi_{\cC_k} \rvert^2 \, \d x \\
 &\qquad\leq C_{d , \mu^{\bullet}} \bigg( \sum_{\ell = 0}^{k - 2} 2^{\frac{d}{2} (\ell - k)} \big( \| \nabla u \|_{\L^2 (\cC_{\ell})} + \| F \|_{\L^2 (\cC_{\ell})} \big) + \sum_{\substack{ \ell \in \IN_0 \\ \lvert \ell - k \rvert \leq 1}} \big( \| \nabla u \|_{\L^2 (\cC_{\ell})} + \| F \|_{\L^2 (\cC_{\ell})} \big) \\
 &\qquad\qquad + \sum_{\ell = k + 2}^{\infty} 2^{(\frac{d}{2} + 1)(k - \ell)} \big( \| \nabla u \|_{\L^2 (\cC_{\ell})} + \| F \|_{\L^2 (\cC_{\ell})} \big) \bigg) \| \phi - \phi_{\cC_k} \|_{\L^2 (\cC_k)}.
\end{align*}
Division by $\| \phi - \phi_{\cC_k} \|_{\L^2 (\cC_k)}$ finally delivers the desired estimate.
\end{proof}

The last lemma gave a control of the pressure by the gradient of $u$ and some parts of the right-hand side of the resolvent equation. The next lemma will contain the standard proof of Caccioppoli's inequality and will provide an estimate of $\lvert \lambda \rvert^{1 / 2} u$ and the gradient of $u$ by $u$, $\lvert \lambda \rvert^{- 1 / 2} f$, and $F$ and also by an arbitrary small pressure term. As in Lemma~\ref{Lem: Non-local pressure estimate} we adopt the notation $\cC_k := B(x_0 , 2^k r) \setminus \overline{B(x_0 , 2^{k - 1} r)}$ for $k \in \IN$.

\begin{lemma}
\label{Lem: Caccioppoli with delta-pressure}
Let $\mu$ satisfy Assumption~\ref{Ass: Coefficients} with constants $\mu^{\bullet} , \mu_{\bullet} > 0$. Then there exists $\omega \in (\pi / 2 , \pi)$ such that for all $\theta \in (0 , \omega)$ there exists $C > 0$ such that for all $\lambda \in \S_{\theta}$, $\delta > 0$, $c \in \IC^d$, $f \in \L^2_{\sigma} (\IR^d)$, and $F \in \L^2 (\IR^d ; \IC^{d \times d})$ the unique solutions $u \in \H^1_{\sigma} (\IR^d)$ and $\phi \in \L^2_{\loc} (\IR^d)$ to
\begin{align*}
 \left\{ \begin{aligned}
  \lambda u - \divergence \mu \nabla u + \nabla \phi &= f + \divergence(F) && \text{in } \IR^d, \\
  \divergence (u) &= 0 && \text{in } \IR^d
 \end{aligned} \right.
\end{align*}
satisfy
\begin{align*}
 &\lvert \lambda \rvert \int_{B(x_0 , 2^k r)} \lvert u \rvert^2 \, \d x + \int_{B(x_0 , 2^k r)} \lvert \nabla u \rvert^2 \, \d x \\
 &\qquad \qquad \leq \delta \int_{\cC_{k + 1}} \lvert \phi - \phi_{\cC_{k + 1}} \rvert^2 \, \d x + C \Big( 1 + \frac{1}{\delta} \Big) \frac{2^{- 2 k}}{r^2} \int_{B(x_0 , 2^{k + 1} r)} \lvert u + c \rvert^2 \, \d x \\
 &\qquad \qquad \qquad + \lvert \lambda \rvert \lvert c \rvert C \int_{B(x_0 , 2^{k + 1} r)} \lvert u \rvert \, \d x + \frac{C}{\lvert \lambda \rvert} \int_{B(x_0 , 2^{k + 1} r)} \lvert f \rvert^2 \, \d x + C \int_{B(x_0 , 2^{k + 1} r)} \lvert F \rvert^2 \, \d x.
\end{align*}
The constant $\omega$ depends only on $\mu_{\bullet}$, $\mu^{\bullet}$, and $d$ and $C > 0$ depends only on $\mu_{\bullet}$, $\mu^{\bullet}$, $d$, and $\theta$.
\end{lemma}

\begin{proof}
Let $\eta \in \C_c^{\infty} (B(x_0 , 2^{k + 1} r))$ with $\eta \equiv 1$ in $B(x_0 , 2^k r)$, $0 \leq \eta \leq 1$, and $\| \nabla \eta \|_{\L^{\infty}} \leq 2 / (2^k r)$. Using $v := \eta^2 (u + c)$ as a test function then delivers
\begin{align}
\label{Eq: Caccioppoli testing}
\begin{aligned}
 \lambda \int_{\IR^d} \lvert u \rvert^2 \eta^2 \, \d x + \lambda \overline{c} \int_{\IR^d} u \eta^2 \, \d x &+ \int_{\IR^d} \mu^{i j}_{\alpha \beta} \partial_{\beta} u_j \partial_{\alpha} [(\overline{u_i} + \overline{c_i}) \eta^2] \, \d x - \int_{\IR^d} \phi \, \divergence(\eta^2 (\overline{u} + \overline{c})) \, \d x \\
 &\qquad = \int_{\IR^d} f \cdot (\overline{u} + \overline{c}) \eta^2 \, \d x - \int_{\IR^d} F_{\alpha \beta} \partial_{\alpha} [(\overline{u_{\beta}} + \overline{c_{\beta}}) \eta^2] \, \d x.
\end{aligned}
\end{align}
First of all, the pressure term can be rewritten as
\begin{align*}
 \int_{\IR^d} \phi \, \divergence(\eta^2 (\overline{u} + \overline{c})) \, \d x = \int_{\IR^d} (\phi - \phi_{\cC_{k + 1}}) \, \divergence(\eta^2 (\overline{u} + \overline{c})) \, \d x = 2 \int_{\IR^d} (\phi - \phi_{\cC_{k + 1}}) \eta \nabla \eta \cdot (\overline{u} + \overline{c}) \, \d x.
\end{align*}
Next, the second-order term can be rewritten as
\begin{align*}
 &\int_{\IR^d} \mu^{i j}_{\alpha \beta} \partial_{\beta} u_j \partial_{\alpha} [(\overline{u_i} + \overline{c_i}) \eta^2] \, \d x \\
 &\qquad= \int_{\IR^d} \mu^{i j}_{\alpha \beta} \partial_{\beta} (u_j + c_j) \partial_{\alpha} [(\overline{u_i} + \overline{c_i}) \eta] \eta \, \d x + \int_{\IR^d} \mu^{i j}_{\alpha \beta} \partial_{\beta} (u_j + c_j) [\partial_{\alpha} \eta] (\overline{u_i} + \overline{c_i}) \eta \, \d x \\
 &\qquad= \int_{\IR^d} \mu^{i j}_{\alpha \beta} \partial_{\beta} [(u_j + c_j) \eta] \partial_{\alpha} [(\overline{u_i} + \overline{c_i}) \eta] \, \d x - \int_{\IR^d} \mu^{i j}_{\alpha \beta} \partial_{\beta} \eta \partial_{\alpha} [(\overline{u_i} + \overline{c_i} ) \eta] (u_j + c_j) \, \d x \\
 &\qquad\qquad + \int_{\IR^d} \mu^{i j}_{\alpha \beta} \partial_{\beta} [(u_j + c_j) \eta] [\partial_{\alpha} \eta] (\overline{u_i} + \overline{c_i}) \, \d x - \int_{\IR^d} \mu^{i j}_{\alpha \beta} \partial_{\beta} \eta [\partial_{\alpha} \eta] (\overline{u_i} + \overline{c_i}) (u_j + c_j) \, \d x \\
 &\qquad=: \mathrm{I} - \mathrm{II} + \mathrm{III} - \mathrm{IV}.
\end{align*}
Rearrange~\eqref{Eq: Caccioppoli testing} in terms of the just derived identities to conclude that
\begin{align*}
 \lambda \int_{\IR^d} \lvert u \rvert^2 \eta^2 \, \d x + \mathrm{I} &= 2 \int_{\IR^d} (\phi - \phi_{\cC_{k + 1}}) \eta \nabla \eta \cdot (\overline{u} + \overline{c}) \, \d x + \mathrm{II} - \mathrm{III} + \mathrm{IV} + \int_{\IR^d} f \cdot (\overline{u} + \overline{c}) \eta^2 \, \d x \\
 &\qquad - \int_{\IR^d} F_{\alpha \beta} \partial_{\alpha} [(\overline{u_{\beta}} + \overline{c_{\beta}}) \eta] \eta \, \d x - \int_{\IR^d} F_{\alpha \beta} [\partial_{\alpha} \eta] (\overline{u_{\beta}} + \overline{c_{\beta}}) \eta \, \d x \\
 &\qquad - \lambda \overline{c} \int_{\IR^d} u \eta^2 \, \d x.
\end{align*}
The ellipticity and boundedness conditions~\eqref{Eq: Ellipticity} and~\eqref{Eq: Boundedness} imply that
\begin{align*}
 \lvert \mathrm{I} \rvert \leq C \Re(\mathrm{I}),
\end{align*}
where $C > 0$ only depends on $\mu_{\bullet}$, $\mu^{\bullet}$, and $d$. This implies the existence of $\omega \in (\pi / 2 , \pi)$ such that
\begin{align*}
 \mathrm{I} \in \overline{\S_{\pi - \omega}}.
\end{align*}
Now, choose $\theta \in (0 , \omega)$. Since $\lambda$ is assumed to be contained in $\S_{\theta}$ and since $\theta + \pi - \omega < \pi$ there exists a constant $C > 0$ depending only on $\theta$ and $\omega$ (and thus only on $\theta$, $\mu_{\bullet}$, $\mu^{\bullet}$, and $d$) such that
\begin{align*}
 \lvert \lambda \rvert \int_{\IR^d} \lvert u \rvert^2 \eta^2 \, \d x + \lvert \mathrm{I} \rvert \leq C \Big\lvert  \lambda \int_{\IR^d} \lvert u \rvert^2 \eta^2 \, \d x + \mathrm{I} \Big\rvert.
\end{align*}
Employing another time the ellipticity condition~\eqref{Eq: Ellipticity} shows that
\begin{align}
\label{Eq: Left-hand side Caccioppoli}
 \lvert \lambda \rvert \int_{\IR^d} \lvert u \rvert^2 \eta^2 \, \d x + \int_{\IR^d} \lvert \nabla [(u + c) \eta] \rvert^2 \, \d x \leq C_{\theta , \mu_{\bullet} , \mu^{\bullet} , d} \Big\lvert  \lambda \int_{\IR^d} \lvert u \rvert^2 \eta^2 \, \d x + \mathrm{I} \Big\rvert,
\end{align}
where $C_{\theta , \mu_{\bullet} , \mu^{\bullet} , d} > 0$ still only depends on $\theta$, $\mu_{\bullet}$, $\mu^{\bullet}$, and $d$. \par
Let $\delta > 0$. The remaining terms are estimated by Young's inequality as
\begin{align}
\label{Eq: Pressure Caccioppoli}
\begin{aligned}
 &2 \Big\lvert \int_{\IR^d} (\phi - \phi_{\cC_{k + 1}}) \eta \nabla \eta \cdot (\overline{u} + \overline{c}) \, \d x \Big\rvert \\
 &\qquad \leq \frac{4 \cdot 2^{- k}}{r} \int_{\cC_{k + 1}} \lvert \phi - \phi_{\cC_{k + 1}} \rvert \lvert (u + c) \eta \rvert \, \d x \\
 &\qquad \leq \frac{\delta}{2 C_{\theta , \mu_{\bullet} , \mu^{\bullet} , d}} \int_{\cC_{k + 1}} \lvert \phi - \phi_{\cC_{k + 1}} \rvert^2 \, \d x + \frac{8 C_{\theta , \mu_{\bullet} , \mu^{\bullet} , d} 2^{- 2 k}}{\delta r^2} \int_{B(x_0 , 2^{k + 1} r)} \lvert u + c \rvert^2 \, \d x
\end{aligned}
\end{align}
and
\begin{align}
\label{Eq: Perturbative term 1}
\begin{aligned}
 \lvert \mathrm{II} - \mathrm{III} \rvert &\leq \frac{4 \cdot 2^{- k} d^4 \mu^{\bullet}}{r} \int_{\cC_{k + 1}} \lvert \nabla [(u + c) \eta] \rvert \lvert u + c \rvert \, \d x \\
 &\leq \frac{1}{4 C_{\theta , \mu_{\bullet} , \mu^{\bullet} , d}} \int_{\IR^d} \lvert \nabla [(u + c) \eta] \rvert^2 \, \d x \\
 &\qquad + \frac{32 \cdot 2^{- 2 k} d^8 (\mu^{\bullet})^2 C_{\theta , \mu_{\bullet} , \mu^{\bullet} , d}}{r^2} \int_{B(x_0 , 2^{k + 1} r)} \lvert u + c \rvert^2 \, \d x
 \end{aligned}
\end{align}
and
\begin{align}
\label{Eq: Perturbative term 2}
 \lvert \mathrm{IV} \rvert \leq \frac{4 \cdot 2^{- 2 k} d^4}{r^2} \int_{B(x_0 , 2^{k + 1} r)} \lvert u + c \rvert^2 \, \d x
\end{align}
and
\begin{align}
\label{Eq: Right-hand side Caccioppoli}
\Big\lvert \int_{\IR^d} f \cdot (\overline{u} + \overline{c} ) \eta^2 \, \d x \Big\rvert \leq \frac{C_{\theta , \mu_{\bullet} , \mu^{\bullet} , d}}{2 \lvert \lambda \rvert} \int_{B(x_0 , 2^{k + 1} r)} \lvert f \rvert^2 \, \d x + \frac{\lvert \lambda \rvert}{2 C_{\theta , \mu_{\bullet} , \mu^{\bullet} , d}} \int_{\IR^d} \lvert (u + c) \eta \rvert^2 \, \d x
\end{align}
and
\begin{align}
\label{Eq: Right-hand side divergence 1}
\begin{aligned}
 &\Big\lvert \int_{\IR^d} F_{\alpha \beta} \partial_{\alpha} [(\overline{u_{\beta}} + \overline{c_{\beta}}) \eta] \eta \, \d x \Big\rvert \\
 &\qquad\leq \frac{1}{4 C_{\theta , \mu_{\bullet} , \mu^{\bullet} , d}} \int_{\IR^d} \lvert \nabla [(u + c) \eta] \rvert^2 \, \d x + C_{\theta , \mu_{\bullet} , \mu^{\bullet} , d} \int_{B(x_0 , 2^{k + 1} r)} \lvert F \rvert^2 \, \d x
 \end{aligned}
\end{align}
and
\begin{align}
\label{Eq: Right-hand side divergence 2}
 \int_{\IR^d} F_{\alpha \beta} [\partial_{\alpha} \eta] (\overline{u_{\beta}} + \overline{c_{\beta}}) \eta \, \d x \leq \frac{4 \cdot 2^{- 2 k}}{r^2} \int_{B(x_0 , 2^{k + 1} r)} \lvert u + c \rvert^2 \, \d x + \int_{B(x_0 , 2^{k + 1} r)} \lvert F \rvert^2 \, \d x.
\end{align}
Putting all the estimates~\eqref{Eq: Left-hand side Caccioppoli},~\eqref{Eq: Pressure Caccioppoli},~\eqref{Eq: Perturbative term 1},~\eqref{Eq: Perturbative term 2},~\eqref{Eq: Right-hand side Caccioppoli},~\eqref{Eq: Right-hand side divergence 1}, and~\eqref{Eq: Right-hand side divergence 2} together and performing a rearrangment of the terms finally delivers
\begin{align*}
 \lvert \lambda \rvert \int_{\IR^d} \lvert u \eta \rvert^2 \, \d x + \int_{\IR^d} \lvert \nabla [(u + c) \eta] \rvert^2 \, \d x &\leq \delta \int_{\cC_{k + 1}} \lvert \phi - \phi_{\cC_{k + 1}} \rvert^2 \, \d x + \lvert \lambda \rvert \lvert c \rvert C \int_{B(x_0 , 2^{k + 1} r)} \lvert u \rvert \, \d x \\
 &\qquad + \Big( 1 + \frac{1}{\delta} \Big) \frac{C \cdot 2^{- 2 k}}{r^2} \int_{B(x_0 , 2^{k + 1} r)} \lvert u + c \rvert^2 \, \d x \\
 &\qquad + \frac{C}{\lvert \lambda \rvert} \int_{B(x_0 , 2^{k + 1} r)} \lvert f \rvert^2 \, \d x + C \int_{B(x_0 , 2^{k + 1} r)} \lvert F \rvert^2 \, \d x,
\end{align*}
where $C > 0$ depends only on $\theta$, $\mu_{\bullet}$, $\mu^{\bullet}$, and $d$. Finally, use that $\eta \equiv 1$ in $B(x_0 , 2^k r)$ and conclude the desired estimate.
\end{proof}

Before we come to the proof of Theorem~\ref{Thm: Non-local Caccioppoli} we state and prove the following elementary lemma.

\begin{lemma}
\label{Lem: Square in series}
Let $0 < \nu < d + 2$ and $(a_{\ell})_{\ell \in \IN_0} \in \ell^{\infty}$. Then there exists a constant $C > 0$ depending only on $d$ and $\nu$ such that
\begin{align*}
 \sum_{k = 1}^{\infty} 2^{- \nu k} \bigg( \sum_{\ell = 0}^{k - 2} 2^{\frac{d}{2} (\ell - k)} a_{\ell} + \sum_{\substack{\ell \in \IN_0 \\ \lvert \ell - k \rvert \leq 1}} a_{\ell} + \sum_{\ell = k + 2}^{\infty} 2^{(\frac{d}{2} + 1) (k - \ell)} a_{\ell} \bigg)^2 \leq C \sum_{\ell = 0}^{\infty} 2^{- \nu \ell} a_{\ell}^2.
\end{align*}
\end{lemma}

\begin{proof}
Let $\vartheta \in (0 , 1)$ satisfy $\vartheta (d + 2) = (d + 2 + \nu) / 2$. By the Cauchy--Schwarz inequality the series is then estimated by
\begin{align*}
 &\sum_{k = 1}^{\infty} 2^{- \nu k} \bigg( \sum_{\ell = 0}^{k - 2} 2^{\frac{d}{2} (\ell - k)} a_{\ell} + \sum_{\substack{\ell \in \IN_0 \\ \lvert \ell - k \rvert \leq 1}} a_{\ell} + \sum_{\ell = k + 2}^{\infty} 2^{(\frac{d}{2} + 1) (k - \ell)} a_{\ell} \bigg)^2 \\
 &\quad \leq 3 \sum_{k = 1}^{\infty} 2^{- \nu k} \bigg\{ \bigg( \sum_{\ell = 0}^{k - 2} 2^{(1 - \vartheta) \frac{d}{2} (\ell - k)} 2^{\vartheta \frac{d}{2} (\ell - k)} a_{\ell} \bigg)^2 + 3 \sum_{\substack{\ell \in \IN_0 \\ \lvert \ell - k \rvert \leq 1}} a_{\ell}^2 \\
 &\quad\qquad + \bigg( \sum_{\ell = k + 2}^{\infty} 2^{(1 - \vartheta) (\frac{d}{2} + 1) (k - \ell)} 2^{\vartheta (\frac{d}{2} + 1) (k - \ell)} a_{\ell} \bigg)^2 \bigg\} \\
 &\quad \leq 3 \sum_{k = 1}^{\infty} 2^{- \nu k} \bigg\{ \sum_{\ell = 0}^{k - 2} 2^{(1 - \vartheta) d (\ell - k)} \cdot \sum_{\ell = 0}^{k - 2} 2^{\vartheta d (\ell - k)} a_{\ell}^2 + 3 \sum_{\substack{\ell \in \IN_0 \\ \lvert \ell - k \rvert \leq 1}} a_{\ell}^2 \\
 &\quad\qquad + \sum_{\ell = k + 2}^{\infty} 2^{(1 - \vartheta) (d + 2) (k - \ell)} \cdot \sum_{\ell = k + 2}^{\infty} 2^{\vartheta (d + 2) (k - \ell)} a_{\ell}^2 \bigg\}\cdotp
\end{align*}
Consequently, there exists a constant $C > 0$ depending only on $d$ and $\nu$ such that
\begin{align*}
 &\sum_{k = 1}^{\infty} 2^{- \nu k} \bigg( \sum_{\ell = 0}^{k - 2} 2^{\frac{d}{2} (\ell - k)} a_{\ell} + \sum_{\substack{\ell \in \IN_0 \\ \lvert \ell - k \rvert \leq 1}} a_{\ell} + \sum_{\ell = k + 2}^{\infty} 2^{(\frac{d}{2} + 1) (k - \ell)} a_{\ell} \bigg)^2 \\
 &\quad \leq C \sum_{k = 1}^{\infty} 2^{- \nu k} \bigg\{ \sum_{\ell = 0}^{k - 2} 2^{\vartheta d (\ell - k)} a_{\ell}^2 + \sum_{\substack{\ell \in \IN_0 \\ \lvert \ell - k \rvert \leq 1}} a_{\ell}^2 + \sum_{\ell = k + 2}^{\infty} 2^{\vartheta (d + 2) (k - \ell)} a_{\ell}^2 \bigg\}\cdotp
\end{align*}
The only terms that are of interest right now are the first and the third series. After applying Fubini's theorem to each of the series we derive by virtue of $\vartheta (d + 2) - \nu > 0$ with a different constant $C > 0$ still depending only on $d$ and $\nu$ that
\begin{align*}
 &\sum_{k = 1}^{\infty} 2^{- \nu k} \bigg\{ \sum_{\ell = 0}^{k - 2} 2^{\vartheta d (\ell - k)} a_{\ell}^2 + \sum_{\ell = k + 2}^{\infty} 2^{\vartheta (d + 2) (k - \ell)} a_{\ell}^2 \bigg\} \\
 &\quad = \sum_{\ell = 0}^{\infty} 2^{\ell \vartheta d} a_{\ell}^2 \sum_{k = \ell + 2}^{\infty} 2^{- (\vartheta d + \nu) k} + \sum_{\ell = 3}^{\infty} 2^{- \vartheta (d + 2) \ell} a_{\ell}^2 \sum_{k = 1}^{\ell - 2} 2^{(\vartheta (d + 2) - \nu ) k} \\
 &\quad\leq C \bigg\{\sum_{\ell = 0}^{\infty} 2^{- \nu \ell} a_{\ell}^2 + \sum_{\ell = 3}^{\infty} 2^{- \nu \ell} a_{\ell}^2 \bigg\}\cdotp \qedhere
\end{align*}
\end{proof}

\begin{proof}[Proof of Theorem~\ref{Thm: Non-local Caccioppoli}]
Let $\delta > 0$ be a constant to be fixed during the proof. Define $\omega \in (\pi / 2 , \pi)$ to be the number determined by Lemma~\ref{Lem: Caccioppoli with delta-pressure}. By virtue of Lemma~\ref{Lem: Caccioppoli with delta-pressure} applied with constant $c = c_k \in \IC^d$ we estimate with a constant $C > 0$ depending only on $\mu_{\bullet}$, $\mu^{\bullet}$, $d$, and $\theta$
\begin{align}
\label{Eq: Iteration}
\begin{aligned}
 &\lvert \lambda \rvert \sum_{k = 0}^{\infty} 2^{- \nu k} \int_{B(x_0 , 2^k r)} \lvert u \rvert^2 \, \d x + \sum_{k = 0}^{\infty} 2^{- \nu k} \int_{B(x_0 , 2^k r)} \lvert \nabla u \rvert^2 \, \d x \\
 &\qquad \leq \delta 2^{\nu} \sum_{k = 1}^{\infty} 2^{- \nu k} \int_{\cC_k} \lvert \phi - \phi_{\cC_k} \rvert^2 \, \d x + \Big( 1 + \frac{1}{\delta} \Big) \frac{C}{r^2} \sum_{k = 0}^{\infty} 2^{- (\nu + 2) k} \int_{B(x_0 , 2^{k + 1} r)} \lvert u + c_k \rvert^2 \, \d x \\
 &\qquad\qquad + \lvert \lambda \rvert \sum_{k = 0}^{\infty} \lvert c_k \rvert 2^{- \nu k} \int_{B(x_0 , 2^{k + 1} r)} \lvert u \rvert \, \d x + \frac{C}{\lvert \lambda \rvert} \sum_{k = 0}^{\infty} 2^{- \nu k} \int_{B(x_0 , 2^{k + 1} r)} \lvert f \rvert^2 \, \d x \\
 &\qquad\qquad + C \sum_{k = 0}^{\infty} 2^{- \nu k} \int_{B(x_0 , 2^{k + 1} r)} \lvert F \rvert^2 \, \d x.
 \end{aligned}
\end{align}
Now, we employ Lemma~\ref{Lem: Non-local pressure estimate} first, followed by Lemma~\ref{Lem: Square in series} with $a_{\ell} := \| \nabla u \|_{\L^2 (\cC_{\ell})} + \| F \|_{\L^2 (\cC_{\ell})}$ for $\ell \in \IN_0$ to establish for some constant $C_{d , \nu} > 0$ that
\begin{align*}
 &\sum_{k = 1}^{\infty} 2^{- \nu k} \int_{\cC_k} \lvert \phi - \phi_{\cC_k} \rvert^2 \, \d x \\
 &\leq \sum_{k = 1}^{\infty} 2^{- \nu k} \bigg( \sum_{\ell = 0}^{k - 2} 2^{\frac{d}{2} (\ell - k)} a_{\ell} + \sum_{\substack{\ell \in \IN_0 \\ \lvert \ell - k \rvert \leq 1}} a_{\ell} + \sum_{\ell = k + 2}^{\infty} 2^{(\frac{d}{2} + 1) (k - \ell)} a_{\ell} \bigg)^2 \\
 &\leq C_{d , \nu} \sum_{\ell = 0}^{\infty} 2^{- \nu \ell} \bigg( \int_{\cC_{\ell}} \lvert \nabla u \rvert^2 \, \d x + \int_{\cC_{\ell}} \lvert F \rvert^2 \, \d x \bigg).
\end{align*}
Plugging this estimate into~\eqref{Eq: Iteration} and using that $\cC_k \subset B(x_0 , 2^k r)$ then delivers
\begin{align*}
 &\lvert \lambda \rvert \sum_{k = 0}^{\infty} 2^{- \nu k} \int_{B(x_0 , 2^k r)} \lvert u \rvert^2 \, \d x + \sum_{k = 0}^{\infty} 2^{- \nu k} \int_{B(x_0 , 2^k r)} \lvert \nabla u \rvert^2 \, \d x \allowdisplaybreaks \\
 &\qquad \leq \delta 2^{\nu} C_{d , \nu} \sum_{k = 0}^{\infty} 2^{- \nu k} \int_{B(x_0 , 2^k r)} \lvert \nabla u \rvert^2 \, \d x + \Big( 1 + \frac{1}{\delta} \Big) \frac{C}{r^2} \sum_{k = 0}^{\infty} 2^{- (\nu + 2) k} \int_{B(x_0 , 2^{k + 1} r)} \lvert u + c_k \rvert^2 \, \d x \\
 &\qquad\qquad + \lvert \lambda \rvert \sum_{k = 0}^{\infty} \lvert c_k \rvert 2^{- \nu k} \int_{B(x_0 , 2^{k + 1} r)} \lvert u \rvert \, \d x + \frac{C}{\lvert \lambda \rvert} \sum_{k = 0}^{\infty} 2^{- \nu k} \int_{B(x_0 , 2^{k + 1} r)} \lvert f \rvert^2 \, \d x \\
 &\qquad\qquad + (C + C_{d , \nu} \delta) \sum_{k = 0}^{\infty} 2^{- \nu k} \int_{B(x_0 , 2^{k + 1} r)} \lvert F \rvert^2 \, \d x.
\end{align*}
Now, choose $\delta$ such that $\delta 2^{\nu} C_{d , \nu} = 1 / 2$. Absorbing the first term on the right-hand side into the corresponding term on the left-hand side finally yields the desired estimate.
\end{proof}

\section{A digression on resolvent estimates and maximal regularity}

\noindent We start this section with a definition.

\begin{definition}
\label{Def: Sectorial and R-sectorial}
Let $X$ and $Y$ denote a Banach space over the complex field and $B : \dom(B) \subset X \to X$ a linear operator. 
\begin{enumerate}
 \item The operator $B$ is said to be \textit{sectorial} of angle $\omega \in (0 , \pi)$ if
 \begin{align*}
  \sigma (B) \subset \overline{\S_{\omega}}
 \end{align*}
 and if for all $0 < \theta < \pi - \omega$ the family $\{ \lambda (\lambda + B)^{-1} \}_{\lambda \in \S_{\theta}} \subset \cL(X)$ is bounded.
 \item A family of operators $\cT \subset \cL (X, Y)$ is said to be \textit{$\cR$-bounded} if there exists a positive constant $C > 0$ such that for any $k_0 \in \IN$, $(T_k)_{k = 1}^{k_0} \subset \cT$, and $(x_k)_{k = 1}^{k_0} \subset X$ the inequality
\begin{align*}
\Big\| \sum_{k = 1}^{k_0} r_k (\cdot) T_k x_k \Big\|_{\L^2 (0 , 1 ; Y)}
\leq C \Big\| \sum_{k = 1}^{k_0} r_k (\cdot) x_k \Big\|_{\L^2 (0 , 1 ; X)}
\end{align*}
holds. Here, $r_k (t) := \mathrm{sgn} (\sin(2^k \pi t))$ are the \textit{Rademacher-functions}. \label{Item: R-boundedness}
 \item The operator $B$ is said to be \textit{$\cR$-sectorial} of angle $\omega \in [0 , \pi)$ if
 \begin{align*}
  \sigma (B) \subset \overline{\S_{\omega}}
 \end{align*}
 and if for all $0 < \theta < \pi - \omega$ the family $\{ \lambda (\lambda + B)^{-1} \}_{\lambda \in \S_{\theta}} \subset \cL(X)$ is $\cR$-bounded.
\end{enumerate}
\end{definition}

\begin{remark}
\label{Rem: R-boundedness}
\begin{enumerate}
 \item By taking $k_0 = 1$ one sees that $\cR$-boundedness implies boundedness of a family of operators. If $X$ and $Y$ are isomorphic to a Hilbert space, then $\cR$-boundedness is equivalent to the boundedness of the family of operators, see~\cite[Rem.~3.2]{Denk_Hieber_Pruess}. \label{Item: R-boundedness implies boundedness}
 \item If $X$ is a subspace of $\L^p (\Omega ; \IC^m)$ for some $1 < p < \infty$, $m \in \IN$, and $\Omega \subset \IR^d$ Lebesgue measurable, then there exists $C > 0$ such that for all $k_0 \in \IN$ and $(f_k)_{k = 1}^{k_0}$ it holds
 \begin{align*}
  \frac{1}{C} \Big\| \sum_{k = 1}^{k_0} r_k (\cdot) f_k \Big\|_{\L^2 (0 , 1 ; X)} \leq \Big\| \Big[ \sum_{k = 1}^{k_0} \lvert f_k \rvert^2 \Big]^{1 / 2} \Big\|_{\L^p (\Omega)} \leq C \Big\| \sum_{k = 1}^{k_0} r_k (\cdot) f_k \Big\|_{\L^2 (0 , 1 ; X)}.
 \end{align*}
 This means, that $\cR$-boundedness in $\L^p$-spaces is equivalent to so-called square function estimates~\cite[Rem.~3.2]{Denk_Hieber_Pruess}. \label{Item: Square function estimate}
 \item The operator $- B$ generates a strongly continuous bounded analytic semigroup on $X$ if and only if $B$ is densely defined and sectorial of angle $\omega \in [0 , \pi / 2)$, see~\cite[Thm.~II.4.6]{Engel_Nagel}. \label{Item: Sectoriality and semigroups}
\end{enumerate}
\end{remark}

If $X$ is a $\mathrm{UMD}$-space, then the question of $\cR$-sectoriality is intimately related to the question of the maximal $\L^q$-regularity of generators of a bounded analytic semigroup $- B : \dom(B) \subset X \to X$. For $1 < q < \infty$, we say that $B$ has maximal $\L^q$-regularity if for all $f \in \L^q (0 , \infty ; X)$ the unique mild solution to the abstract Cauchy problem
\begin{align*}
 \left\{ \begin{aligned}
  u^{\prime} (t) + B u (t) &= f (t), && t > 0, \\
  u(0) &= 0
 \end{aligned} \right.
\end{align*}
which is given by Duhamel's formula
\begin{align*}
 u (t) = \int_0^t \e^{- (t - s) B} f(s) \, \d s
\end{align*}
satisfies for almost every $t > 0$ that $u(t) \in \dom(B)$ and $B u \in \L^q (0 , \infty ; X)$. In this case $u$ is also weakly differentiable with respect to $t$ and satisfies $u^{\prime} \in \L^q (0 , \infty ; X)$. By employing the closed graph theorem, the mere fact that $u^{\prime}$ and $B u$ lie in $\L^q (0 , \infty ; X)$ implies the existence of $C > 0$ such that for all $f \in \L^q (0 , \infty ; X)$ the stability estimate
\begin{align*}
 \| u^{\prime} \|_{\L^q (0 , \infty ; X)} + \| B u \|_{\L^q (0 , \infty ; X)} \leq C \| f \|_{\L^q (0 , \infty ; X)}
\end{align*}
holds. Notice further, that the property that $B$ has maximal $\L^q$-regularity is \textit{independent} of $p$, see~\cite[Thm.~7.1]{Dore}. \par
A seminal result of Weis~\cite[Thm.~4.2]{Weis} now builds the bridge between the notions of maximal $\L^q$-regularity and $\cR$-sectoriality.

\begin{theorem}[Weis]
\label{Thm: Weis}
If $X$ is a $\mathrm{UMD}$-space and $- B : \dom(B) \subset X \to X$ the generator of a bounded analytic semigroup, then $B$ has maximal $\L^q$-regularity for $1 < q < \infty$ if and only if $B$ is $\cR$-sectorial of angle $\omega \in [0 , \pi / 2)$.
\end{theorem}

We now turn our focus towards the Stokes operator. We state the following proposition concerning the required $\L^2$-theory. Its proof is omitted as it relies on a simple application of the Lax--Milgram lemma and of using the solution $u$ as a test function to the resolvent equation.

\begin{proposition}
\label{Prop: L2 properties}
Let $\mu$ satisfy Assumption~\ref{Ass: Coefficients} for some constants $\mu_{\bullet} , \mu^{\bullet} > 0$. Then there exists $\omega \in (\pi / 2 , \pi)$ depending on $d$, $\mu_{\bullet}$, and $\mu^{\bullet}$ such that $\S_{\omega} \subset \rho(- A)$ and for all $\theta \in (0 , \omega)$ there exists $C > 0$ such that for all $\lambda \in \S_{\theta}$ and all $f \in \L^2_{\sigma} (\IR^d)$ it holds with $u := (\lambda + A)^{-1} f$ that
\begin{align*}
 \lvert \lambda \rvert \| u \|_{\L^2_{\sigma}} + \lvert \lambda \rvert^{\frac{1}{2}} \| \nabla u \|_{\L^2} \leq C \| f \|_{\L^2_{\sigma}}.
\end{align*}
Moreover, for all $\lambda \in \S_{\theta}$ and all $F \in \C_c^{\infty} (\IR^d ; \IC^{d \times d})$ there exists $C > 0$ such that with $u := (\lambda + A)^{-1} \IP \divergence(F)$ it holds
\begin{align*}
 \lvert \lambda \rvert^{\frac{1}{2}} \| u \|_{\L^2_{\sigma}} + \| \nabla u \|_{\L^2} \leq C \| F \|_{\L^2}.
\end{align*}
The constants $C$ only depend on $d$, $\mu_{\bullet}$, $\mu^{\bullet}$, and $\theta$.
\end{proposition}

\begin{observation}
\label{Obs: Square function estimate}
In the situation of Proposition~\ref{Prop: L2 properties} we see that the family $\{ \lambda (\lambda + A)^{-1} \}_{\lambda \in \S_{\theta}}$ is bounded for all $\theta \in (0 , \omega)$. Thus, $A$ regarded as an operator on $\L^2_{\sigma} (\IR^d)$ is sectorial of angle $\omega$. Moreover, by Remark~\ref{Rem: R-boundedness}~\eqref{Item: R-boundedness implies boundedness} the operator $A$ is also $\cR$-sectorial of angle $\omega$. \par
A combination of Definition~\ref{Def: Sectorial and R-sectorial}~\eqref{Item: R-boundedness} and Remark~\ref{Rem: R-boundedness}~\eqref{Item: Square function estimate} reveals that the following statement is immediate: \par
For $1 < p < \infty$, the family $\{ \lambda (\lambda + A)^{-1} \}_{\lambda \in \S_{\theta}}$ extends from $\C_{c , \sigma}^{\infty} (\IR^d)$ to an $\cR$-bounded family in $\cL(\L^p_{\sigma} (\IR^d))$ if and only if there exists $C > 0$ such that for all $k_0 \in \IN$, all $(\lambda_k)_{k = 1}^{k_0} \subset \S_{\theta}$, and all $f := (f_1 , \dots , f_{k_0}, 0 \dots)$ with $f_k \in \C_{c , \sigma}^{\infty} (\IR^d)$, $1 \leq k \leq k_0$, the operator $T_{(\lambda_k)_{k = 1}^{k_0}}$ given by
\begin{align*}
 T_{(\lambda_k)_{k = 1}^{k_0}} f \mapsto \begin{pmatrix} \lambda_1 (\lambda_1 + A)^{-1} f_1 \\ \vdots \\ \lambda_{k_0} (\lambda_{k_0} + A)^{-1} f_{k_0} \\ 0 \\ \vdots \end{pmatrix}
\end{align*}
satisfies
\begin{align}
\label{Eq: Vector-valued boundedness}
 \| T_{(\lambda_k)_{k = 1}^{k_0}} f \|_{\L^p (\IR^d ; \ell^2 (\IC^d))} \leq C \| f \|_{\L^p (\IR^d ; \ell^2 (\IC^d))}.
\end{align}
In other words, the family of all operators that can be formed by the procedure above extends to a bounded family in $\L^p (\IR^d ; \ell^2 (\IC^d))$. In particular, from the first part of this observation, we know that this statement is valid in the case $p = 2$.
\end{observation}

\section{A glimpse onto a non-local $\L^p$-extrapolation theorem}
\label{Sec: A glimpse onto a non-local Lp-extrapolation theorem}

\noindent To extrapolate the $\cR$-bounded family $\{ \lambda (\lambda + A)^{-1} \}_{\lambda \in \S_{\theta}}$ for $\theta \in (0 , \omega)$ in $\Lop(\L^2_{\sigma} (\IR^d))$ to an $\cR$-bounded family in $\cL(\L^p_{\sigma} (\IR^d))$ for $p$ satisfying~\eqref{Eq: Lp interval} we want to employ the following vector valued and non-local analogue of Shen's $\L^p$-extrapolation theorem~\cite[Thm.~1.2]{Tolksdorf_nonlocal}. See~\cite[Thm.~3.1]{Shen} for Shen's original theorem.

\begin{theorem}
\label{Thm: Non-local, vector valued Shen}
Let $X$, $Y$, and $Z$ be Banach spaces, $\cM , \cN > 0$, and let 
\begin{align*}
 T \in \Lop(\L^2(\IR^d ; X) , \L^2(\IR^d ; Y)) \quad &\text{with} \quad \| T \|_{\Lop(\L^2(\IR^d ; X) , \L^2(\IR^d ; Y))} \leq \cM \\
  \text{and} \quad \cC \in \Lop(\L^2(\IR^d ; X) , \L^2(\IR^d ; Z)) \quad &\text{with} \quad \| \cC \|_{\Lop(\L^2(\IR^d ; X) , \L^2(\IR^d ; Y))} \leq \cN.
\end{align*}

\indent Suppose that there exist constants $p_0 > 2$, $\iota > 1$, and $C > 0$ such that for all balls $B \subset \IR^d$ and all compactly supported $f \in \L^{\infty}(\IR^d ; X)$ with $f = 0$ in $\iota B$ the estimate
\begin{align}
\label{Eq: Weak reverse Hoelder inequality}
 \begin{aligned}
 \bigg( \fint_{B} \| T f \|_Y^{p_0} \; \d x \bigg)^{\frac{1}{p_0}} &\leq C \sup_{B^{\prime} \supset B} \bigg( \fint_{B^{\prime}} \big( \| T f \|_Y^2 + \| \cC f \|_Z^2 \big) \; \d x \bigg)^{\frac{1}{2}}
 \end{aligned}
\end{align}
holds. Here the supremum runs over all balls $B^{\prime}$ containing $B$. \par
Then for each $2 < p < p_0$ there exists a constant $K > 0$ such that for all $f \in \L^{\infty} (\IR^d ; X)$ with compact support it holds
\begin{align*}
 \| T f \|_{\L^p (\IR^d ; Y)} \leq K \big( \| f \|_{\L^p (\IR^d ; X)} + \| \cC f \|_{\L^p (\IR^d ; Z)} \big).
\end{align*}
In particular, if $\cC$ is bounded from $\L^p (\IR^d ; X)$ into $\L^p (\IR^d ; Z)$, then the restriction of $T$ onto $\L^2(\IR^d ; X) \cap \L^p(\IR^d ; X)$ extends to a bounded linear operator from $\L^p (\IR^d ; X)$ into $\L^p(\IR^d ; Y)$. The constant $K$ depends only on $d$, $p_0$, $p$, $\iota$, $C$, $\cM$, and $\cN$.
\end{theorem}

In our situation, we would like to choose $X = Y = Z := \ell^2 (\IC^d)$, $\cC := \Id$, and the operator $T$ as one of the operators defined in Observation~\ref{Obs: Square function estimate}. If all these operators would satisfy the assumptions of Theorem~\ref{Thm: Non-local, vector valued Shen} in a uniform manner, we could conclude the $\cR$-boundedness of this family in $\L^p$-spaces. However, there is one issue, namely, the resolvent operators $(\lambda + A)^{-1}$ are only defined on $\L^2_{\sigma} (\IR^d)$ and not on $\L^2 (\IR^d ; \IC^d)$. Clearly, one could try to replace $(\lambda + A)^{-1}$ by the operator $(\lambda + A)^{-1} \IP$, which is a bounded operator defined on all of $\L^2 (\IR^d ; \IC^d)$. However, as it was mentioned earlier, to verify~\eqref{Eq: Weak reverse Hoelder inequality} one needs Caccioppoli's inequality, cf.~Theorem~\ref{Thm: Non-local Caccioppoli}, and this inequality requires the right-hand side $f$ to be solenoidal. More precisely, the solenoidality was essentially used in the proof of Lemma~\ref{Lem: Non-local pressure estimate}. Being bound to right-hand sides in $\L^2_{\sigma} (\IR^d)$ we have to have a closer look onto the proof of Theorem~\ref{Thm: Non-local, vector valued Shen} that can be found in~\cite[Thm.~1.2]{Tolksdorf_nonlocal}. Notice that a similar analysis was performed in~\cite[Sect.~5]{Tolksdorf_convex} for the scalar valued case. \par
Throughout the proof of Theorem~\ref{Thm: Non-local, vector valued Shen} a function $f$ is fixed and for exactly this function the boundedness estimate
\begin{align*}
  \| T f \|_{\L^p (\IR^d ; Y)} \leq K \big( \| f \|_{\L^p (\IR^d ; X)} + \| \cC f \|_{\L^p (\IR^d ; Z)} \big)
\end{align*}
is proved. To establish this estimate a good-$\lambda$ argument is used. An analysis of this good-$\lambda$ argument reveals that the $\L^2$-boundedness of $T$ and $\cC$ as well as~\eqref{Eq: Weak reverse Hoelder inequality} are used exactly once, namely, in order to deduce an inequality of the form
\begin{align}
\label{Eq: Good lambda key}
\begin{aligned}
 &\lvert \{ x \in Q : M_{2 Q^*} (\| T f \|_Y^2) (x) > \alpha \} \rvert \leq \frac{C}{\alpha} \int_{2 \iota Q^*} \big(\| f \|_X^2 + \| \cC f \|_Z^2 \big) \; \d x \\
 &\qquad\qquad\qquad\qquad+ \frac{C \lvert Q \rvert}{\alpha^{p_0 / 2}} \bigg\{ \sup_{Q^{\prime} \supset 2 Q^*} \bigg( \frac{1}{\lvert Q^{\prime} \rvert} \int_{Q^{\prime}} \big( \| T f \|_Y^2 + \| f \|_X^2 + \| \cC f \|_Z^2 \big) \; \d x \bigg)^{\frac{1}{2}} \bigg\}^{p_0},
\end{aligned}
\end{align}
cf.\@ the proof of Claim~3 in~\cite[Thm.~1.2]{Tolksdorf_convex}. Here, $\alpha > 0$ is arbitrary, $Q$ is a cube in $\IR^d$, $Q^*$ is its dyadic ``parent'', i.e., $Q$ arises from $Q^*$ by bisecting its sides, and $M_{2 Q^*}$ is the localized maximal operator
\begin{align*}
 M_{2 Q^*} g (x) := \sup_{\substack{x \in R \\ R \subset 2 Q^*}}\frac{1}{\lvert R \rvert} \int_R \lvert g \rvert \; \d y \qquad (x \in 2 Q^*),
\end{align*}
where in the supremum $R$ denotes a cube in $\IR^d$. To derive~\eqref{Eq: Good lambda key} from~\eqref{Eq: Weak reverse Hoelder inequality} and the $\L^2$-boundedness of $T$ and $\cC$, notice that~\eqref{Eq: Weak reverse Hoelder inequality} can equivalently be formulated with cubes instead of balls. Then, $f$ is decomposed as $f = f \chi_{2 \iota Q^*} + f \chi_{\IR^d \setminus 2 \iota Q^*}$, where $\chi$ denotes the characteristic function of a set. This decomposition is used on the left-hand side of~\eqref{Eq: Good lambda key} to estimate
\begin{align}
\label{Eq: Decomposition by sharp cutoff}
\begin{aligned}
 \lvert \{ x \in Q : M_{2 Q^*} (\| T f \|_Y^2) (x) > \alpha \} \rvert &\leq \lvert \{ x \in Q : M_{2 Q^*} (\| T f \chi_{2 \iota Q^*} \|_Y^2) (x) > \alpha / 4 \} \rvert \\
 &\qquad + \lvert \{ x \in Q : M_{2 Q^*} (\| T f \chi_{\IR^d \setminus 2 \iota Q^*} \|_Y^2) (x) > \alpha / 4 \} \rvert.
\end{aligned}
\end{align}
The first term on the right-hand side is controlled by the weak type-$(1 , 1)$ estimate of the localized maximal operator and the $\L^2$-boundedness of $T$, yielding the first term on the right-hand side of~\eqref{Eq: Good lambda key}. The second term on the right-hand side is controlled by the embedding $\L^{p_0 / 2} \hookrightarrow \L^{p_0 / 2 , \infty}$ and the $\L^{p_0 / 2}$-boundedness of the localized maximal operator followed by~\eqref{Eq: Weak reverse Hoelder inequality} and the $\L^2$-boundedness of $T$ and $\cC$ yielding the remaining terms on the right-hand side of~\eqref{Eq: Good lambda key}, cf.\@ the proof of Claim~3 in~\cite[Thm.~1.2]{Tolksdorf_convex}. \par
Essentially, the only thing that happened in~\eqref{Eq: Decomposition by sharp cutoff} was that $T f$ was decomposed by means of
\begin{align}
\label{Eq: Decomposition of Tf}
 T f =  T f \chi_{2 \iota Q^*} +  T f \chi_{\IR^d \setminus 2 \iota Q^*}.
\end{align}
We would like to emphasize here that this decomposition of $T f$ is induced by the linearity of $T$ and a decomposition of $f$. Clearly, one could imagine that other suitable decompositions of $T f$ into a sum of two functions exist and that these might not have anything to do with a decomposition of $f$. Taking this into account in the formulation of the $\L^p$-extrapolation theorem might yield a more flexible result. This could be an advantage if a certain structure of $f$ (such as solenoidality) is eminent and which is destroyed by multiplication by characteristic functions. This indicates the need of a formulation of Shen's $\L^p$-extrapolation theorem that does not rely on a particular decomposition of $T f$ and is presented in the following. \par
To this end, we say that $Q^*$ is the parent of a cube $Q \subset \IR^d$ if $Q$ arises from $Q^*$ by bisecting its sides. Moreover, for $x_0 \in \IR^d$ and $r > 0$ let $Q(x_0 , r)$ denote the cube in $\IR^d$ with center $x_0$ and $\diam(Q (x_0 , r)) = r$. Finally, for a number $\alpha > 0$ denote by $\alpha Q$ the cube $Q (x_0 , \alpha r)$. In the following formulation of the $\L^p$-extrapolation theorem, we simply replace the $\L^2$-boundedness of $T$ and $\cC$ together with~\eqref{Eq: Weak reverse Hoelder inequality} by the assumption that~\eqref{Eq: Good lambda key} is valid.

\begin{theorem}
\label{Thm: Modified Shen whole space}
Let $X$, $Y$, and $Z$ be Banach spaces. Let further $2 < p < p_0$, $f \in \L^2 (\IR^d ; X) \cap \L^p (\IR^d ; X)$, and let $T$ be an operator (not necessarily linear) such that $T(f)$ is defined and contained in $\L^2 (\IR^d ; Y)$. \par
Suppose that there exist constants $\iota > 1$ and $C > 0$ and an operator $\cC$ (not necessarily linear) such that $\cC(f)$ is defined lies in $\L^2 (\IR^d ; Z)$ such that for all $\alpha > 0$, all $Q = Q(x_0 , r)$ with $r > 0$ and $x_0 \in \IR^d$, and all parents $Q^*$ of $Q$ the estimate
\begin{align}
\label{Eq: Generalized weak reverse Hoelder inequality IR^d}
 \begin{aligned}
 &\lvert \{ x \in Q : M_{2 Q^*} (\| T (f) \|_Y^2) (x) > \alpha \} \rvert \leq \frac{C}{\alpha} \int_{2 \iota Q^*} \big(\| f \|_X^2 + \| \cC (f) \|_Z^2 \big) \; \d x \\
 &\qquad\qquad\qquad\qquad+ \frac{C \lvert Q \rvert}{\alpha^{p_0 / 2}} \bigg\{ \sup_{Q^{\prime} \supset 2 Q^*} \bigg( \frac{1}{\lvert Q^{\prime} \rvert} \int_{Q^{\prime}} \big( \| T (f) \|_Y^2 + \| f \|_X^2 + \| \cC (f) \|_Z^2 \big) \; \d x \bigg)^{\frac{1}{2}} \bigg\}^{p_0},
 \end{aligned}
\end{align}
holds. Here the supremum runs over all cubes $Q^{\prime}$ containing $2 Q^*$. \par
Then there exists a constant $K > 0$ depending on $d$, $p$, $p_0$, $\iota$, and $C$ such that
\begin{align*}
 \| T(f) \|_{\L^p (\IR^d ; Y)} \leq K \big( \| f \|_{\L^p (\IR^d ; X)} + \| \cC (f) \|_{\L^p (\IR^d ; Z)} \big).
\end{align*}
\end{theorem}

\section{Proofs of Theorems~\ref{Thm: Resolvent},~\ref{Thm: Gradient}, and~\ref{Thm: Max Reg}}

\noindent This section is dedicated to the proofs of Theorems~\ref{Thm: Resolvent},~\ref{Thm: Gradient}, and~\ref{Thm: Max Reg}. Notice that by virtue of Remark~\ref{Rem: R-boundedness}~\eqref{Item: R-boundedness implies boundedness} and Definition~\ref{Def: Sectorial and R-sectorial} the statement of Theorem~\ref{Thm: Resolvent} is a mere corollary of Theorem~\ref{Thm: Max Reg}. Thus, we will focus only on the proofs of Theorems~\ref{Thm: Gradient} and~\ref{Thm: Max Reg}. \par
Before we delve into the proofs of these results, we have another look onto a Caccioppoli inequality which is similar to the one in Lemma~\ref{Lem: Caccioppoli with delta-pressure}.

\begin{lemma}
\label{Lem: Caccioppoli variable pressure}
Let $\mu$ satisfy Assumption~\ref{Ass: Coefficients} for some constants $\mu_{\bullet} , \mu^{\bullet} > 0$. Let $\omega \in (\pi / 2 , \pi)$ be the number provided by Lemma~\ref{Lem: Caccioppoli with delta-pressure}. Then for all $\theta \in (0 , \omega)$ there exists $C > 0$ such that for all $\lambda \in \S_{\theta}$ and all solutions $u \in \H^1 (B(x_0 , 2 r))$ and $\phi \in \L^2 (B(x_0 , 2 r))$ (in the sense of distributions) to
\begin{align*}
 \left\{ \begin{aligned}
  \lambda u - \divergence \mu \nabla u + \nabla \phi &= 0 && \text{in } B(x_0 , 2 r), \\
  \divergence (u) &= 0 && \text{in } B(x_0 , 2 r)
 \end{aligned} \right.
\end{align*}
satisfy for all $\eta \in \C_c^{\infty} (B(x_0 , 2 r))$ with $\eta \equiv 1$ in $B(x_0 , r)$, $0 \leq \eta \leq 1$, and $\| \nabla \eta \|_{\L^{\infty}} \leq 2 / r$ and for all $c_1 \in \IC$ and $c_2 \in \IC^d$ the estimate
\begin{align*}
 &\lvert \lambda \rvert \int_{B(x_0 , 2 r)} \lvert u \eta \rvert^2 \, \d x + \int_{B(x_0 , 2 r)} \lvert \nabla [(u + c_2) \eta] \rvert^2 \, \d x \leq \frac{C}{r^2} \int_{B(x_0 , 2 r)} \lvert u + c_2 \rvert^2 \, \d x \\
 &\quad + \lvert c_2 \rvert \lvert \lambda \rvert \int_{B(x_0 , 2 r)} \lvert u \rvert \eta^2 \, \d x + \frac{4}{r} \bigg(\int_{B(x_0 , 2 r) \setminus B(x_0 , r)} \lvert \phi - c_1 \rvert^2 \, \d x \bigg)^{\frac{1}{2}} \bigg( \int_{B(x_0 , 2 r)} \lvert (u + c_2) \eta \rvert^2 \, \d x \bigg)^{\frac{1}{2}}.
\end{align*}
The constant $C > 0$ depends only on $\mu_{\bullet}$, $\mu^{\bullet}$, $d$, and $\theta$.
\end{lemma}

\begin{proof}
The proof is literally the same as the proof of Lemma~\ref{Lem: Caccioppoli with delta-pressure} in the case $k = 0$. The only difference is how one estimates the arising pressure term in~\eqref{Eq: Pressure Caccioppoli}. Notice that in order to derive~\eqref{Eq: Pressure Caccioppoli} we subtracted the constant $\phi_{\cC_{k + 1}}$ of $\phi$ but that it was possible to subtract any other constant as well. Thus, it is no problem to replace $\phi_{\cC_{k + 1}}$ by $c_1$ in~\eqref{Eq: Pressure Caccioppoli}. This term then reads
\begin{align*}
2 \Big\lvert \int_{\IR^d} (\phi - c_1) \eta \nabla \eta \cdot (\overline{u} + \overline{c_2}) \, \d x \Big\rvert.
\end{align*}
Now, by the properties of $\eta$ and by H\"older's inequality we find that
\begin{align*}
 &\Big\lvert \int_{\IR^d} (\phi - c_1) \eta \nabla \eta \cdot (\overline{u} + \overline{c_2}) \, \d x \Big\rvert \\
 &\qquad \leq \frac{2}{r} \bigg(\int_{B(x_0 , 2 r) \setminus B(x_0 , r)} \lvert \phi - c_1 \rvert^2 \, \d x \bigg)^{\frac{1}{2}} \bigg( \int_{B(x_0 , 2 r)} \lvert (u + c_2) \eta \rvert^2 \, \d x \bigg)^{\frac{1}{2}}.
\end{align*}
This readily concludes the proof.
\end{proof}

To proceed we introduce another sesquilinear form, which is connected to the Stokes problem in a ball but with Neumann boundary conditions. For this purpose, let $B \subset \IR^d$ denote a ball and let
\begin{align*}
 \cL^2_{\sigma} (B) := \{ f \in \L^2 (B ; \IC^d) : \divergence(f) = 0 \text{ in the sense of distributions} \}
\end{align*}
and let
\begin{align*}
 \cH^1_{\sigma} (B) := \{ f \in \H^1 (B ; \IC^d) : \divergence(f) = 0 \}.
\end{align*}
Now, define the sesquilinear form
\begin{align*}
 \fb_B : \cH^1_{\sigma} (B) \times \cH^1_{\sigma} (B) \to \IC, \quad (u , v) \mapsto \int_{\IR^d} \mu^{i j}_{\alpha \beta} \partial_{\beta} u_j \overline{\partial_{\alpha} v_i} \; \d x.
\end{align*}
We abuse the notation and denote the same sesquilinear form but with domain $\H^1 (B ; \IC^d) \times \H^1 (B ; \IC^d)$ again by $\fb_B$.

\begin{remark}
Let $\theta \in (0 , \omega)$, $\lambda \in \S_{\theta}$, and $B \subset \IR^d$ be a ball. Let further $f \in \cL^2_{\sigma} (B)$ and $F \in \L^2 (B ; \IC^{d \times d})$ and let $u \in \cH^1_{\sigma} (B)$ be a solution to
\begin{align}
\label{Eq: Solenoidal distributional Neumann}
 \lambda \int_B u \cdot \overline{v} \, \d x + \fb_B (u , v) = \int_B f \cdot \overline{v} \, \d x - \int_B F_{\alpha \beta} \overline{\partial_{\alpha} v_{\beta}} \, \d x \quad (v \in \cH^1_{\sigma} (B)).
\end{align}
First of all, testing by $u$ implies the existence of a constant $C > 0$ depending only on $d$, $\theta$, $\mu_{\bullet}$, and $\mu^{\bullet}$ such that
\begin{align}
\label{Eq: Resolvent Neumann}
 \lvert \lambda \rvert \| u \|_{\cL^2_{\sigma} (B)} + \lvert \lambda \rvert^{\frac{1}{2}} \| \nabla u \|_{\L^2 (B)} \leq C \big( \| f \|_{\cL^2_{\sigma} (B)} + \lvert \lambda \rvert^{\frac{1}{2}} \| F \|_{\L^2 (B)} \big).
\end{align}
Second, since $\C_{c , \sigma}^{\infty} (B) \subset \cH^1_{\sigma} (B)$ there exists by virtue of~\cite[Lem.~II.2.2.2]{Sohr} a pressure function $\vartheta \in \L^2 (B)$, which is unique up the the addition of constants, such that
\begin{align}
\label{Eq: Distributional Neumann}
 \lambda \int_B u \cdot \overline{v} \, \d x + \fb_B (u , v) - \int_B \vartheta \, \overline{\divergence(v)} \, \d x = \int_B f \cdot \overline{v} \, \d x - \int_B F_{\alpha \beta} \overline{\partial_{\alpha} v_{\beta}} \, \d x \quad (v \in \H^1_0 (B ; \IC^d)).
\end{align}
Now, we show in the following that one can even find a constant $c \in \IC$ such that with $\phi := \vartheta + c$ one has
\begin{align}
\label{Eq: Variational Neumann}
 \lambda \int_B u \cdot \overline{v} \, \d x + \fb_B (u , v) - \int_B \phi \, \overline{\divergence(v)} \, \d x = \int_B f \cdot \overline{v} \, \d x - \int_B F_{\alpha \beta} \overline{\partial_{\alpha} v_{\beta}} \, \d x \quad (v \in \H^1 (B ; \IC^d)).
\end{align}

\indent To prove the existence of $c$ we repeat the argument of~\cite[Pf.~of~Thm.~6.8]{Mitrea_Monniaux_Wright} in the case of constant coefficients. Let $\varphi_0 \in \H^{1 / 2} (\partial B ; \IC^d)$ with
\begin{align*}
 \int_{\partial B} \frac{x - x_0}{\lvert x - x_0 \rvert} \cdot \varphi_0 \, \d \sigma (x) = 1,
\end{align*}
where $x_0$ denotes the center of $B$ and $\sigma$ its surface measure. Let $E \varphi_0 \in \H^1 (B  ;\IC^d)$ denote an extension of $\varphi_0$ and define
\begin{align*}
 c := - \int_B f \cdot \overline{E \varphi_0} \, \d x + \int_B F_{\alpha \beta} \overline{\partial_{\alpha} (E \varphi_0)_{\beta}} \, \d x + \lambda \int_B u \cdot \overline{E \varphi_0} \, \d x + \fb_B (u , E \varphi_0) - \int_B \vartheta \, \overline{\divergence(E \varphi_0)} \, \d x.
\end{align*}
Now, for $v \in \H^1 (B ; \IC^d)$ we find with
\begin{align*}
 \eta := \int_{\partial B} \frac{x - x_0}{\lvert x - x_0 \rvert} \cdot v|_{\partial B} \, \d \sigma (x).
\end{align*}
that
\begin{align*}
 &\lambda \int_B u \cdot \overline{v} \, \d x + \fb_B (u , v) - \int_B (\vartheta + c) \, \overline{\divergence(v)} \, \d x \\
 &\qquad= \lambda \int_B u \cdot \overline{v} \, \d x + \fb_B (u , v) - \int_B \vartheta \, \overline{\divergence(v)} \, \d x - c \eta \\
 &\qquad = \lambda \int_B u \cdot (\overline{v - \eta E \varphi_0}) \, \d x + \fb_B (u , v - \eta E \varphi_0) - \int_B \vartheta \, \overline{\divergence(v - \eta E \varphi_0)} \, \d x \\
 &\qquad \qquad - \int_B f \cdot \overline{(v - \eta E \varphi_0)} \, \d x + \int_B F_{\alpha \beta} \overline{\partial_{\alpha} (v - \eta E \varphi_0)_{\beta}} \, \d x \\
 &\qquad\qquad + \int_B f \cdot \overline{v} \, \d x - \int_B F_{\alpha \beta} \overline{\partial_{\alpha} v_{\beta}} \, \d x.
\end{align*}
By virtue of~\cite[Lem.~2.3]{Mitrea_Monniaux_Wright} the trace operator
\begin{align*}
 \tr : \cH^1_{\sigma} (B) \to \bigg\{ g \in \H^{\frac{1}{2}} (\partial B ; \IC^d) : \int_{\partial B} \frac{x - x_0}{\lvert x - x_0 \rvert} \cdot g \, \d \sigma (x) = 0 \bigg\} =: \H^{\frac{1}{2}}_n (\partial B).
\end{align*}
is onto. By construction, we have that $v - \eta E \varphi_0 \in \H^{1 / 2}_n (\partial B)$. Thus, there exists $\psi \in \cH^1_{\sigma} (B)$ with $\tr(\psi) = v - \eta E \varphi_0$. In particular, we have that $v - \eta E \varphi_0 - \psi \in \H^1_0 (\Omega ; \IC^d)$. Thus, employing first~\eqref{Eq: Distributional Neumann} and then~\eqref{Eq: Solenoidal distributional Neumann} delivers
\begin{align*}
 &\lambda \int_B u \cdot (\overline{v - \eta E \varphi_0}) \, \d x + \fb_B (u , v - \eta E \varphi_0) - \int_B \vartheta \, \overline{\divergence(v - \eta E \varphi_0)} \, \d x \\
 &\qquad \qquad - \int_B f \cdot \overline{(v - \eta E \varphi_0)} \, \d x + \int_B F_{\alpha \beta} \overline{\partial_{\alpha} (v - \eta E \varphi_0)_{\beta}} \, \d x \\
 &= \lambda \int_B u \cdot \overline{\psi} \, \d x + \fb_B (u , \psi) - \int_B f \cdot \overline{\psi} \, \d x + \int_B F_{\alpha \beta} \overline{\partial_{\alpha} \psi_{\beta}} \, \d x \\
 &= 0.
\end{align*}
This establishes~\eqref{Eq: Variational Neumann}. Having~\eqref{Eq: Variational Neumann} at our disposal, we can also derive a bound on the pressure function $\phi$. Indeed, testing~\eqref{Eq: Variational Neumann} with $v := \nabla \Delta_D^{-1} \phi$ (where $\Delta_D$ denotes the Dirichlet Laplacian on $B$), using that $u$ and $f$ are orthogonal to $v$, and using~\eqref{Eq: Resolvent Neumann} delivers with a constant $C > 0$ depending only on $d$, $\theta$, $\mu_{\bullet}$, and $\mu^{\bullet}$
\begin{align}
\label{Eq: Pressure Neumann}
 \lvert \lambda \rvert^{\frac{1}{2}} \| \phi \|_{\L^2 (B)} \leq C \big( \| f \|_{\cL^2_{\sigma} (B)} + \lvert \lambda \rvert^{\frac{1}{2}} \| F \|_{\L^2 (B)} \big).
\end{align}
\end{remark}

As described in Section~\ref{Sec: A glimpse onto a non-local Lp-extrapolation theorem} we want to study $u := (\lambda + A)^{-1} f$ and - in order to verify~\eqref{Eq: Generalized weak reverse Hoelder inequality IR^d} - we want to find suitable decompositions of $u$ into $u = u_1 + u_2$ which should be valid in $2 Q^*$ for a given cube $Q$. This is done in the following lemma. The argument to arrive at the desired estimate is subtle. We will apply Lemma~\ref{Lem: Caccioppoli variable pressure} with $c_2 = 0$ and use that we left the term on the right-hand side involving the pressure, i.e.,
\begin{align*}
 \frac{4}{r} \bigg(\int_{2 B \setminus B} \lvert \phi - c_1 \rvert^2 \, \d x \bigg)^{\frac{1}{2}} \bigg( \int_{2 B} \lvert u \eta \rvert^2 \, \d x \bigg)^{\frac{1}{2}}
\end{align*}
in a product structure. In this situation, one can still decide whether one estimates the term by Young's inequality as
\begin{align*}
 \frac{4}{r} \bigg(\int_{2 B \setminus B} \lvert \phi - c_1 \rvert^2 \, \d x \bigg)^{\frac{1}{2}} \bigg( \int_{2 B} \lvert u \eta \rvert^2 \, \d x \bigg)^{\frac{1}{2}} \leq \frac{1}{2} \int_{2 B \setminus B} \lvert \phi - c_1 \rvert^2 \, \d x + \frac{8}{r^2} \int_{2 B} \lvert u \eta \rvert^2 \, \d x
\end{align*}
or for some suitable $\eps > 0$ as
\begin{align*}
 \frac{4}{r} \bigg(\int_{2 B \setminus B} \lvert \phi - c_1 \rvert^2 \, \d x \bigg)^{\frac{1}{2}} \bigg( \int_{2 B} \lvert u \eta \rvert^2 \, \d x \bigg)^{\frac{1}{2}} \leq \frac{8}{\eps r^2 \lvert \lambda \rvert} \int_{2 B \setminus B} \lvert \phi - c_1 \rvert^2 \, \d x + \frac{\lvert \lambda \rvert \eps}{2} \int_{2 B} \lvert u \eta \rvert^2 \, \d x.
\end{align*}
In the first situation, one leaves the term involving $u$ on the right-hand side and in the second situation, one can absorb this term onto the left-hand side. Depending on the particular situation, we will need to decide differently.

\begin{lemma}
\label{Lem: Preparation of reverse Holder}
Let $\mu$ satisfy Assumption~\ref{Ass: Coefficients} with constants $\mu^{\bullet} , \mu_{\bullet} > 0$. Let $\omega \in (\pi / 2 , \pi)$ be the number determined by Lemma~\ref{Lem: Caccioppoli with delta-pressure}. Then for any $0 < \theta < \omega$ the following holds: \par
Let $f \in \L^2_{\sigma} (\IR^d)$, $F \in \L^2 (\IR^d ; \IC^{d \times d})$, and let $\lambda \in \S_{\theta}$. For $u \in \H^1_{\sigma} (\IR^d)$ defined by $u := (\lambda + A)^{-1} (f + \IP \divergence(F))$ and $x_0 \in \IR^d$ and $r_0 > 0$ there exists a decomposition of $u$ of the form $u = u_1 + u_2$ with $u_1 \in \cH^1_{\sigma} (B(x_0 , r_0))$ and $u_2 \equiv u$ in $\IR^d \setminus B(x_0 , r_0)$ and there exists $\phi_1 \in \L^2 (B(x_0 , r_0))$ and $C > 0$ such that for any ball $B \subset \IR^d$ of radius $r > 0$ with $2B \subset B(x_0 , r_0)$ we have
\begin{align}
\label{Eq: Second estimate}
\begin{aligned}
 &\lvert \lambda \rvert^3 r^2 \int_{B} \lvert u_2 \rvert^2 \, \d x + \lvert \lambda \rvert^2 r^2 \int_{B} \lvert \nabla u_2 \rvert^2 \; \d x \\
  &\quad \leq C \bigg\{ \sum_{\ell = 0}^{\infty} 2^{- \ell d - \ell} \int_{2^{\ell} B} \big(\lvert \lambda u \rvert^2 + \lvert f \rvert^2 + \lvert \lvert \lambda \rvert^{\frac{1}{2}} F \rvert^2\big) \, \d x + \int_{2 B} \lvert \lambda u_1 \rvert^2 \, \d x + \int_{2 B} \lvert \lvert \lambda \rvert^{\frac{1}{2}} \phi_1 \rvert^2 \, \d x \bigg\}\cdotp
\end{aligned}
\end{align}
Moreover, $u_1$ and $\phi_1$ satisfy for some $C > 0$
\begin{align}
\label{Eq: First estimate u}
\begin{aligned}
 \lvert \lambda \rvert \| u_1 \|_{\L^2 (B(x_0 , r_0))} + \lvert \lambda \rvert^{\frac{1}{2}} \| \nabla u_1 \|_{\L^2 (B (x_0 , r_0))} &+ \lvert \lambda \rvert^{\frac{1}{2}} \| \phi_1 \|_{\L^2 (B(x_0 , r_0))} \\
 &\quad\leq C \big( \| f \|_{\L^2 (B (x_0 , r_0))} + \lvert \lambda \rvert^{\frac{1}{2}} \| F \|_{\L^2 (B(x_0 , r_0))} \big).
\end{aligned}
\end{align}
In both inequalities, the constant $C$ does only depend on $d$, $\theta$, $\mu_{\bullet}$, and $\mu^{\bullet}$.
\end{lemma}

\begin{proof}
Fix $f \in \L^2_{\sigma} (\IR^d)$, $F \in \L^2 (\IR^d ; \IC^{d \times d})$, and $\lambda \in \S_{\theta}$. Define $u := (\lambda + A)^{-1} (f + \IP \divergence(F))$ and let $\phi \in \L^2_{\loc} (\IR^d)$ be the associated pressure. Let $B \subset \IR^d$ be a ball of radius $r > 0$ with $2 B \subset B(x_0 , r_0)$ and let $g := f|_{B(x_0 , r_0)}$. The definition of $\Lop^2_{\sigma} (B(x_0 , r_0))$ implies that $g \in \Lop^2_{\sigma} (B(x_0 , r_0))$. Let further $G := F|_{B(x_0 , r_0)}$. Then, there exists $u_1 \in \cH^1_{\sigma} (B(x_0 , r_0))$ such that for all $v \in \cH^1_{\sigma} (B(x_0 , r_0))$ it holds
\begin{align*}
 \lambda \int_{B(x_0 , r_0)} u_1 \cdot \overline{v} \; \d x + \fb_{B(x_0 , r_0)} (u_1 , v) = \int_{B(x_0 , r_0)} g \cdot \overline{v} \; \d x - \int_{B(x_0 , r_0)} G_{\alpha \beta} \cdot \overline{\partial_{\alpha} v_{\beta}} \; \d x.
\end{align*}
Let $\phi_1 \in \L^2 (B(x_0 , r_0))$ denote the associated pressure. By~\eqref{Eq: Resolvent Neumann} and~\eqref{Eq: Pressure Neumann} we find that
\begin{align*}
 \lvert \lambda \rvert \| u_1 \|_{\L^2 (B(x_0 , r_0))} + \lvert \lambda \rvert^{\frac{1}{2}} \| \nabla u_1 \|_{\L^2 (B (x_0 , r_0))} &+ \lvert \lambda \rvert^{\frac{1}{2}} \| \phi_1 \|_{\L^2 (B(x_0 , r_0))} \\
 &\qquad\leq C \big( \| f \|_{\L^2 (B (x_0 , r_0))} + \lvert \lambda \rvert^{\frac{1}{2}} \| F \|_{\L^2 (B(x_0 , r_0))} \big).
\end{align*}
Notice that the constant $C > 0$ depends only on $d$, $\theta$, $\mu_{\bullet}$, and $\mu^{\bullet}$. In particular, it does not depend on $r_0$.  \par
Now, define $u_2 := u - \widetilde{u_1}$ and $\phi_2 := \phi - \widetilde{\phi_1}$. Here, $\widetilde{u_1}$ and $\widetilde{\phi_1}$ denote the extensions by zero to all of $\IR^d$ of $u_1$ and $\phi_1$. By definitions of all functions, we find that
\begin{align*}
 \lambda \int_{B (x_0 , r_0)} u_2 \cdot \overline{v} \; \d x + \fb_{B (x_0 , r_0)} (u_2 , v) - \int_{B (x_0 , r_0)} \phi_2 \, \overline{\divergence(v)} \; \d x = 0 \qquad (v \in \H^1_0 (B (x_0 , r_0) ; \IC^d)).
\end{align*}
Let $\eta \in \C_c^{\infty} (2 B)$ with $\eta \equiv 1$ in $B$, $0 \leq \eta \leq 1$, and $\| \nabla \eta \|_{\L^{\infty}} \leq 2 / r$. We apply Lemma~\ref{Lem: Caccioppoli variable pressure} with $c_1 \in \IC$ and $c_2 = 0$ to $u_2$ and $\phi_2$ leading to the estimate
\begin{align*}
 &\lvert \lambda \rvert^3 r^2 \int_{2 B} \lvert u_2 \eta \rvert^2 \, \d x + \lvert \lambda \rvert^2 r^2 \int_{2 B} \lvert \nabla [u_2 \eta] \rvert^2 \; \d x \\
 &\qquad \leq 4 \lvert \lambda \rvert^2 r \bigg(\int_{2 B \setminus B} \lvert \phi_2 - c_1 \rvert^2 \, \d x \bigg)^{\frac{1}{2}} \bigg( \int_{2 B} \lvert u_2 \eta \rvert^2 \, \d x \bigg)^{\frac{1}{2}} + C \lvert \lambda \rvert^2 \int_{2 B} \lvert u_2 \rvert^2 \, \d x \\
 &\qquad \leq 4 \lvert \lambda \rvert^2 r \bigg(\int_{2 B \setminus B} \lvert \phi - c_1 \rvert^2 \, \d x \bigg)^{\frac{1}{2}} \bigg( \int_{2 B} \lvert u_2 \eta \rvert^2 \, \d x \bigg)^{\frac{1}{2}} + C \lvert \lambda \rvert^2 \int_{2 B} \lvert u_2 \rvert^2 \, \d x \\
 &\qquad\qquad + 4 \lvert \lambda \rvert^2 r \bigg(\int_{2 B \setminus B} \lvert \phi_1 \rvert^2 \, \d x \bigg)^{\frac{1}{2}} \bigg( \int_{2 B} \lvert u_2 \eta \rvert^2 \, \d x \bigg)^{\frac{1}{2}}.
\end{align*}
Set $c_1 := \phi_{2 B \setminus B}$ and apply Lemma~\ref{Lem: Non-local pressure estimate} with $k = 1$ to estimate $\phi - \phi_{2 B \setminus B}$ in the first inequality. In the second, use H\"older's inequality for series and in the third, employ Theorem~\ref{Thm: Non-local Caccioppoli} with $\nu = d + 1$ and $c_{\ell} = 0$. This yields
\begin{align*}
 &\lvert \lambda \rvert^3 r^2 \int_{2 B} \lvert u_2 \eta \rvert^2 \, \d x + \lvert \lambda \rvert^2 r^2 \int_{2 B} \lvert \nabla [u_2 \eta] \rvert^2 \; \d x \\
&\qquad \leq 4 \lvert \lambda \rvert^2 r \bigg\{ \bigg(\sum_{\substack{\ell \in \IN_0 \\ \lvert \ell - 1 \rvert \leq 1}} (\| \nabla u \|_{\L^2 (\cC_{\ell})} + \| F \|_{\L^2 (\cC_{\ell})} + \sum_{\ell = 3}^{\infty} 2^{(\frac{d}{2} + 1) (1 - \ell)} (\| \nabla u \|_{\L^2 (\cC_{\ell})} + \| F \|_{\L^2 (\cC_{\ell})})\bigg) \\
 &\qquad \qquad \cdot \bigg( \int_{2 B} \lvert u_2 \eta \rvert^2 \, \d x \bigg)^{\frac{1}{2}} \bigg\} + C \lvert \lambda \rvert^2 \int_{2 B} \lvert u_2 \rvert^2 \, \d x + 4 \lvert \lambda \rvert^2 r \bigg(\int_{2 B \setminus B} \lvert \phi_1 \rvert^2 \, \d x \bigg)^{\frac{1}{2}} \bigg( \int_{2 B} \lvert u_2 \eta \rvert^2 \, \d x \bigg)^{\frac{1}{2}} \allowdisplaybreaks \\
 &\qquad \leq C_d \lvert \lambda \rvert^2 r \bigg(\sum_{\ell = 0}^{\infty} 2^{- \ell d - \ell} \int_{2^{\ell} B}\lvert \nabla u \rvert^2 \, \d x\bigg)^{\frac{1}{2}} \bigg( \int_{2 B} \lvert u_2 \eta \rvert^2 \, \d x \bigg)^{\frac{1}{2}} \\
 &\qquad\qquad + C_d \lvert \lambda \rvert^2 r \bigg(\sum_{\ell = 0}^{\infty} 2^{- \ell d - \ell} \int_{2^{\ell} B}\lvert F \rvert^2 \, \d x\bigg)^{\frac{1}{2}} \bigg( \int_{2 B} \lvert u_2 \eta \rvert^2 \, \d x \bigg)^{\frac{1}{2}} \\
 &\qquad\qquad + C \lvert \lambda \rvert^2 \int_{2 B} \lvert u_2 \rvert^2 \, \d x + 4 \lvert \lambda \rvert^2 r \bigg(\int_{2 B \setminus B} \lvert \phi_1 \rvert^2 \, \d x \bigg)^{\frac{1}{2}} \bigg( \int_{2 B} \lvert u_2 \eta \rvert^2 \, \d x \bigg)^{\frac{1}{2}} \allowdisplaybreaks \\
 &\qquad \leq C_{d , \theta, \mu^{\bullet} , \mu_{\bullet}} \bigg(\sum_{\ell = 0}^{\infty} 2^{- \ell d - 3 \ell} \int_{2^{\ell} B}\lvert \lambda u \rvert^2 \, \d x \bigg)^{\frac{1}{2}} \bigg( \int_{2 B} \lvert \lvert \lambda \rvert u_2 \rvert^2 \, \d x \bigg)^{\frac{1}{2}} \\
 &\qquad\qquad + C_{d , \theta, \mu^{\bullet} , \mu_{\bullet}} \bigg(\sum_{\ell = 0}^{\infty} 2^{- \ell d - \ell} \int_{2^{\ell} B}\lvert f \rvert^2 \, \d x \bigg)^{\frac{1}{2}} \bigg( \lvert \lambda \rvert^3 r^2 \int_{2 B} \lvert u_2 \eta \rvert^2 \, \d x \bigg)^{\frac{1}{2}} \\
  &\qquad\qquad + C_{d , \theta , \mu^{\bullet} , \mu_{\bullet}} \bigg(\sum_{\ell = 0}^{\infty} 2^{- \ell d - \ell} \int_{2^{\ell} B}\lvert \lvert \lambda \rvert^{\frac{1}{2}} F \rvert^2 \, \d x\bigg)^{\frac{1}{2}} \bigg( \lvert \lambda \rvert^3  r^2\int_{2 B} \lvert u_2 \eta \rvert^2 \, \d x \bigg)^{\frac{1}{2}} \\
 &\qquad\qquad + C \int_{2 B} \lvert \lvert \lambda \rvert u_2 \rvert^2 \, \d x + 4 \bigg(\int_{2 B \setminus B} \lvert \lvert \lambda \rvert^{\frac{1}{2}} \phi_1 \rvert^2 \, \d x \bigg)^{\frac{1}{2}} \bigg( \lvert \lambda \rvert^3 r^2 \int_{2 B} \lvert u_2 \eta \rvert^2 \, \d x \bigg)^{\frac{1}{2}}.
\end{align*}
Consequently, by Young's inequality and the properties of $\eta$ we find that
\begin{align*}
 &\lvert \lambda \rvert^3 r^2 \int_{2 B} \lvert u_2 \eta \rvert^2 \, \d x + \lvert \lambda \rvert^2 r^2 \int_{2 B} \lvert \nabla [u_2 \eta] \rvert^2 \; \d x \\
 &\quad \leq \frac{C_{d , \theta, \mu^{\bullet} , \mu_{\bullet}}}{2} \sum_{\ell = 0}^{\infty} 2^{- \ell d - 3 \ell} \int_{2^{\ell} B}\lvert \lambda u \rvert^2 \, \d x + C \int_{2 B} \lvert \lvert \lambda \rvert u_2 \rvert^2 \, \d x  + C_{d , \theta, \mu^{\bullet} , \mu_{\bullet}}^2 \sum_{\ell = 0}^{\infty} 2^{- \ell d - \ell} \int_{2^{\ell} B}\lvert f \rvert^2 \, \d x \\
 &\quad\qquad + C_{d , \theta, \mu^{\bullet} , \mu_{\bullet}}^2 \sum_{\ell = 0}^{\infty} 2^{- \ell d - \ell} \int_{2^{\ell} B}\lvert \lvert \lambda \rvert^{\frac{1}{2}} F \rvert^2 \, \d x  + 16 \int_{2 B \setminus B} \lvert \lvert \lambda \rvert^{\frac{1}{2}} \phi_1 \rvert^2 \, \d x + \frac{3}{4} \lvert \lambda \rvert^3 r^2 \int_{2 B} \lvert u_2 \eta \rvert^2 \, \d x.
\end{align*}
Now, the last term on the right-hand side can be absorbed onto the left-hand side. Moreover, we use that $u_2 = u - \widetilde{u_1}$ and we estimate $2^{- \ell d - 3 \ell} \leq 2^{- \ell d - \ell}$, so that we find a constant $C > 0$ depending on $d$, $\theta$, $\mu^{\bullet}$, and $\mu_{\bullet}$ that
\begin{align*}
 &\lvert \lambda \rvert^3 r^2 \int_{2 B} \lvert u_2 \eta \rvert^2 \, \d x + \lvert \lambda \rvert^2 r^2 \int_{2 B} \lvert \nabla [u_2 \eta] \rvert^2 \; \d x \\
 &\qquad \leq C \bigg\{ \sum_{\ell = 0}^{\infty} 2^{- \ell d - \ell} \int_{2^{\ell} B} \big(\lvert \lambda u \rvert^2 + \lvert f \rvert^2 + \lvert \lvert \lambda \rvert^{\frac{1}{2}} F \rvert^2 \big) \, \d x + \int_{2 B}\lvert \lambda u_1 \rvert^2 \, \d x + \int_{2 B} \lvert \lvert \lambda \rvert^{\frac{1}{2}} \phi_1 \rvert^2 \, \d x \bigg\}\cdotp
\end{align*}
Since $\eta \equiv 1$ in $B$ we conclude the estimate~\eqref{Eq: Second estimate}.
\end{proof}

We are now in the position to present the proofs of Theorems~\ref{Thm: Resolvent} and~\ref{Thm: Max Reg}.

\begin{proof}[Proofs of Theorems~\ref{Thm: Resolvent} and~\ref{Thm: Max Reg}]
As it was mentioned earlier, Theorem~\ref{Thm: Resolvent} is a direct consequence of Theorem~\ref{Thm: Max Reg}. This can also be seen in the proof as Theorem~\ref{Thm: Resolvent} corresponds to taking $k_0 = 1$ throughout. \par

\subsection*{The case $p > 2$}

Let $k_0 \in \IN$, $\omega \in (\pi / 2 , \pi)$ as in Lemma~\ref{Lem: Caccioppoli with delta-pressure}. For $\theta \in (0 , \omega)$ let $(\lambda_k)_{k = 1}^{k_0} \subset \S_{\theta}$ and let $(f_k)_{k = 1}^{k_0} \subset \C_{c , \sigma}^{\infty} (\IR^d)$. For $f = (f_1 , \dots , f_{k_0} , 0 , \dots)$, we saw in Observation~\ref{Obs: Square function estimate} that we need to prove the estimate
\begin{align*}
 \| T_{(\lambda_k)_{k = 1}^{k_0}} f \|_{\L^p (\IR^d ; \ell^2 (\IC^d))} \leq C \| f \|_{\L^p (\IR^d ; \ell^2 (\IC^d))},
\end{align*}
with a constant being uniform with respect to all choices above. This will be done by verifying the assumptions of Theorem~\ref{Thm: Modified Shen whole space} uniformly with respect to the choices of parameters above. For the application of Theorem~\ref{Thm: Modified Shen whole space} we set $X = Y = Z = \ell^2 (\IC^d)$ and let $\cC = \Id$. \par
For this purpose, define for $1 \leq k \leq k_0$ the functions $u_k := (\lambda_k + A)^{-1} f_k$. Let $x_0 \in \IR^d$, $r > 0$, and let $B := B(x_0 , r)$. Let further $u_{k , 1}$, $u_{k , 2}$, and $\phi_{k , 1}$ denote the functions provided by Lemma~\ref{Lem: Preparation of reverse Holder} with $r_0 := 2 r$. Notice that
\begin{align}
\label{Eq: Decomposition of T}
 \| T_{(\lambda_k)_{k = 1}^{k_0}} f \|_{\ell^2} = \Big[ \sum_{k = 1}^{k_0} \lvert \lambda_k u_k \rvert^2 \Big]^{\frac{1}{2}} \leq \Big[ \sum_{k = 1}^{k_0} \lvert \lambda_k u_{k , 1} \rvert^2 \Big]^{\frac{1}{2}} + \Big[ \sum_{k = 1}^{k_0} \lvert \lambda_k u_{k , 2} \rvert^2 \Big]^{\frac{1}{2}}\cdotp
\end{align}
Moreover, in $B$ we have for $1 \leq j \leq d$ by the chain rule and Cauchy--Schwarz' inequality that
\begin{align}
\label{Eq: Derivative of square function}
 \Big\lvert \partial_j \Big[ \sum_{k = 1}^{k_0} \lvert \lambda_k u_{k , 2} \rvert^2 \Big]^{\frac{1}{2}} \Big\rvert &= \Big\lvert \Big[ \sum_{k = 1}^{k_0} \lvert \lambda_k u_{k , 2} \rvert^2 \Big]^{- \frac{1}{2}} \sum_{k = 1}^{k_0} \lvert \lambda_k \rvert \Re(\overline{u_{k , 2}} \partial_j u_{k , 2}) \Big\rvert \leq \Big[ \sum_{k = 1}^{k_0} \lvert \lambda_k \partial_j u_{k , 2} \rvert^2 \Big]^{\frac{1}{2}}\cdotp
\end{align}
We start by deriving some kind of non-local weak reverse H\"older inequality for the second term of the right-hand side of~\eqref{Eq: Decomposition of T}. If $d = 2$ let $p_0 > 2$ and if $d \geq 3$ let $p_0 := 2d / (d - 2)$. By Sobolev's embedding theorem together with~\eqref{Eq: Derivative of square function}, there exists a constant $C_{d , p_0} > 0$ depending only on $d$ and $p_0$ such that
\begin{align*}
 &\bigg( \fint_B \Big[ \sum_{k = 1}^{k_0} \lvert \lambda_k u_{k , 2} \rvert^2 \Big]^{\frac{p_0}{2}} \; \d x \bigg)^{\frac{1}{p_0}} \\
 &\leq C \bigg\{ \bigg( \fint_B \Big[ \sum_{k = 1}^{k_0} \lvert \lambda_k u_{k , 2} \rvert^2 \Big]^{\frac{2}{2}} \; \d x \bigg)^{\frac{1}{2}} + \bigg( \sum_{k = 1}^{k_0} \lvert \lambda_k \rvert^2 r^2 \fint_B \lvert \nabla u_{k , 2} \rvert^2 \; \d x \bigg)^{\frac{1}{2}} \bigg\}\cdotp
\end{align*}
To estimate the second term on the right-hand side, let $\eta \in \C_c^{\infty} (2 B)$ with $\eta \equiv 1$ in $B$, $0 \leq \eta \leq 1$, and $\| \nabla \eta \|_{\L^{\infty}} \leq 2 / r$. Observe that 
\begin{align*}
 \fint_B \lvert \nabla u_{k , 2} \rvert^2 \, \d x \leq 2^d \fint_{2 B} \lvert \nabla [u_{k , 2} \eta ] \rvert^2 \, \d x.
\end{align*}
Thus, employing~\eqref{Eq: Second estimate} with $F = 0$ in the first inequality and the decomposition $u_k = u_{k , 1} + u_{k , 2}$ together with~\eqref{Eq: First estimate u} in the second delivers for some constant $C > 0$ depending only on $d$, $\theta$, $\mu_{\bullet}$, $\mu^{\bullet}$, and $p_0$ that
\begin{align}
\label{Eq: Reverse Holder for u2}
\begin{aligned}
 &\bigg( \fint_B \Big[ \sum_{k = 1}^{k_0} \lvert \lambda_k u_{k , 2} \rvert^2 \Big]^{\frac{p_0}{2}} \; \d x \bigg)^{\frac{1}{p_0}} \\
 &\qquad\leq C \bigg\{ \bigg( \fint_B \Big[ \sum_{k = 1}^{k_0} \lvert \lambda_k u_{k , 2} \rvert^2 \Big]^{\frac{2}{2}} \; \d x \bigg)^{\frac{1}{2}} + \bigg(\sum_{\ell = 0}^{\infty} 2^{- \ell} \fint_{2^{\ell} B} \sum_{k = 1}^{k_0} \big( \lvert \lambda_k u_k \rvert^2 + \lvert f_k \rvert^2 \big) \, \d x \\
 &\qquad\qquad + \fint_{2 B} \sum_{k = 1}^{k_0} \big(\lvert \lvert \lambda_k \rvert^{\frac{1}{2}} \phi_{k , 1} \rvert^2 + \lvert \lambda_k u_{k , 1} \rvert^2 \big) \, \d x \bigg)^{\frac{1}{2}} \bigg\} \\
 &\qquad \leq C \sum_{\ell = 0}^{\infty} 2^{- \ell} \fint_{2^{\ell} B} \sum_{k = 1}^{k_0} \big( \lvert \lambda_k u_k \rvert^2 + \lvert f_k \rvert^2 \big) \, \d x.
\end{aligned}
\end{align}
Now, we are in the position to verify the assumptions of Theorem~\ref{Thm: Modified Shen whole space}. Let $Q = Q (x_0 , r / 18)$ be a cube in $\IR^d$ with center $x_0$ and $\diam(Q) = r / 18$ and notice that $2 Q^* \subset B(x_0 , r / 6)$, where $Q^*$ denotes a of parent of $Q$. Then for any $\alpha > 0$ we find by virtue of~\eqref{Eq: Decomposition of T} that
\begin{align}
\label{Eq: Distributions function}
\begin{aligned}
 \lvert\{ x \in Q : M_{2 Q^*} (\| T f \|_{\ell^2}^2) (x) > \alpha \} \rvert &\leq \bigg\lvert\bigg\{ x \in Q : M_{2 Q^*} \Big(\sum_{k = 1}^{k_0} \lvert \lambda_k u_{k , 1} \rvert^2 \Big) (x) > \frac{\alpha}{4} \bigg\} \bigg\rvert \\
 &\qquad+ \bigg\lvert\bigg\{ x \in Q : M_{2 Q^*} \Big( \Big[ \sum_{k = 1}^{k_0} \lvert \lambda_k u_{k , 2} \rvert^2 \Big]^{\frac{2}{2}} \Big) (x) > \frac{\alpha}{4} \bigg\} \bigg\rvert\cdotp
\end{aligned}
\end{align} 
The weak-$(1 , 1)$ estimate of the localized maximal operator followed by~\eqref{Eq: First estimate u} directly yields with some constant depending only on $d$, $\theta$, $\mu_{\bullet}$, and $\mu^{\bullet}$
\begin{align*}
 \bigg\lvert\bigg\{ x \in Q : M_{2 Q^*} \Big(\sum_{k = 1}^{k_0} \lvert \lambda_k u_{k , 1} \rvert^2\Big) (x) > \frac{\alpha}{4} \bigg\} \bigg\rvert \leq \frac{C}{\alpha} \int_{Q (x_0 , 2 \sqrt{d} r)} \Big[ \sum_{k = 1}^{k_0} \lvert f_k \rvert^2 \Big]^{\frac{2}{2}} \, \d x.
\end{align*}
For the second term in~\eqref{Eq: Distributions function} use the weak-$(p_0 / 2 , p_0 / 2)$ inequality of the localized maximal operator followed by~\eqref{Eq: Reverse Holder for u2}. This gives with a constant $C > 0$ depending only on $d$, $\theta$, $\mu_{\bullet}$, $\mu^{\bullet}$, and $p_0$ that
\begin{align*}
 &\bigg\lvert\bigg\{ x \in Q : M_{2 Q^*} \Big( \Big[ \sum_{k = 1}^{k_0} \lvert \lambda_k u_{k , 2} \rvert^2 \Big]^{\frac{2}{2}} \Big) (x) > \frac{\alpha}{4} \bigg\} \bigg\rvert \\
 &\qquad\leq \frac{C}{\alpha^{p_0 / 2}} \bigg( \fint_{B(x_0 , r / 6)} \Big[ \sum_{k = 1}^{k_0} \lvert \lambda_k u_{k , 2} \rvert^2 \Big]^{\frac{p_0}{2}} \; \d x \bigg)^{\frac{1}{p_0}} \\
 &\qquad\leq \frac{C}{\alpha^{p_0 / 2}} \sum_{\ell = 0}^{\infty} 2^{- \ell} \bigg( \fint_{B(x_0 , 2^{\ell} r)}\Big[ \sum_{k = 1}^{k_0} \lvert \lambda_k u_k \rvert^2 \Big]^{\frac{2}{2}} + \Big[ \sum_{k = 1}^{k_0} \lvert f_k \rvert^2 \Big]^{\frac{2}{2}} \, \d x \bigg)^{\frac{1}{2}} \\
  &\qquad\leq \frac{C}{\alpha^{p_0 / 2}} \sup_{Q^{\prime} \supset 2 Q^*} \bigg( \fint_{Q^{\prime}} \Big[ \sum_{k = 1}^{k_0} \lvert \lambda_k u_k \rvert^2 \Big]^{\frac{2}{2}} + \Big[ \sum_{k = 1}^{k_0} \lvert f_k \rvert^2 \Big]^{\frac{2}{2}} \, \d x \bigg)^{\frac{1}{2}}.
\end{align*}
This concludes the proof of this case.

\subsection*{The case $p < 2$}

This case follows directly by the duality principle as described by Kalton and Weis in~\cite[Lem.~3.1]{Kalton_Weis} since $\L^p_{\sigma} (\IR^d)$ is of non-trivial Rademacher type if $1 < p < \infty$.
\end{proof}

\begin{proof}[Proof of Theorem~\ref{Thm: Gradient}]
Let $\omega \in (\pi / 2 , \pi)$ as in Lemma~\ref{Lem: Caccioppoli with delta-pressure}. For $\theta \in (0 , \omega)$ let $\lambda \in \S_{\theta}$. We argue by duality and prove the $\L^p$-boundedness of 
\begin{align*}
 T := \lvert \lambda \rvert^{1 / 2} (\lambda + A)^{-1} \IP \divergence
\end{align*}
for $p \geq 2$ satisfying~\eqref{Eq: Lp interval}. The uniform bound follows by verifying the assumptions of Theorem~\ref{Thm: Modified Shen whole space} uniformly with respect to $\lambda$. \par
We choose $X = Y = Z = \IC^{d \times d}$ and $\cC = \Id$. Let $F \in \L^2 (\IR^d ; \IC^{d \times d})$ and define $u := (\lambda + A)^{-1} \IP \divergence(F)$. Let further $x_0 \in \IR^d$ and $r > 0$ and let $u_1$, $u_2$, and $\phi_1$ denote the corresponding functions from Lemma~\ref{Lem: Preparation of reverse Holder} with $r_0 := 2 r$. Let $p_0 > 2$ if $d = 2$ and $p_0 := 2d / (d - 2)$ if $d \geq 3$. Then, Sobolev's inequality implies that
\begin{align*}
 \bigg( \fint_{B(x_0 , r)} \lvert \lvert \lambda \rvert^{\frac{1}{2}} u_2 \rvert^{p_0} \, \d x \bigg)^{\frac{1}{p_0}} \leq C \bigg\{ \bigg( \fint_{B(x_0 , r)} \lvert \lvert \lambda \rvert^{\frac{1}{2}} u_2 \rvert^2 \, \d x \bigg)^{\frac{1}{2}} + \bigg( r^2 \fint_{B(x_0 , r)} \lvert \lvert \lambda \rvert^{\frac{1}{2}} \nabla u_2 \rvert^2 \, \d x \bigg)^{\frac{1}{2}} \bigg\}.
\end{align*}
Now, let $\eta \in \C_c^{\infty} (2 B)$ with $\eta \equiv 1$ in $B$, $0 \leq \eta \leq 1$, and $\| \nabla \eta \|_{\L^{\infty}} \leq 2 / r$. Then
\begin{align*}
 r^2 \fint_{B(x_0 , r)} \lvert \lvert \lambda \rvert^{\frac{1}{2}} \nabla u_2 \rvert^2 \, \d x \leq 2^d \lvert \lambda \rvert r^2 \fint_{B(x_0 , 2 r)} \lvert \nabla [\eta u_2] \rvert^2 \, \d x.
\end{align*}
Now, employ~\eqref{Eq: Second estimate} with $f = 0$ in the first inequality and then~\eqref{Eq: First estimate u} with $f = 0$ to get
\begin{align*}
 \bigg( \fint_{B(x_0 , r)} \lvert \lvert \lambda \rvert^{\frac{1}{2}} u_2 \rvert^{p_0} \, \d x \bigg)^{\frac{1}{p_0}} &\leq C \bigg\{ \bigg( \fint_{B(x_0 , r)} \lvert \lvert \lambda \rvert^{\frac{1}{2}} u_2 \rvert^2 \, \d x \bigg)^{\frac{1}{2}} + \bigg( \sum_{\ell = 0}^{\infty} 2^{- \ell d - \ell} \int_{2^{\ell} B} \big(\lvert \lvert \lambda \rvert^{\frac{1}{2}} u \rvert^2 + \lvert F \rvert^2\big) \, \d x \\
 &\qquad + \int_{2 B} \lvert \lvert \lambda \rvert^{\frac{1}{2}} u_1 \rvert^2 \, \d x + \int_{2 B} \lvert \phi_1 \rvert^2 \, \d x \bigg)^{\frac{1}{2}} \bigg\} \\
 &\leq C \bigg( \sum_{\ell = 0}^{\infty} 2^{- \ell d - \ell} \int_{2^{\ell} B} \big(\lvert \lvert \lambda \rvert^{\frac{1}{2}} u \rvert^2 + \lvert F \rvert^2\big) \, \d x \bigg)^{\frac{1}{2}}\cdotp
\end{align*}
The rest of the proof can be finished literally as the proof of Theorem~\ref{Thm: Max Reg} starting from~\eqref{Eq: Distributions function}.
\end{proof}

\end{document}